\documentclass[11pt, a4paper, reqno]{amsart}

\usepackage{enumerate}
\usepackage{amssymb}
\usepackage{color}
\usepackage{graphics}
\usepackage{epsfig}

\newtheorem{lemma}{Lemma}[section]
\newtheorem{corollary}[lemma]{Corollary}
\newtheorem{theorem}[lemma]{Theorem}
\newtheorem{proposition}[lemma]{Proposition}
\newtheorem{remark}[lemma]{Remark}

\newtheorem{example}[lemma]{Example}

\def\Sn{\operatorname{Sn}}
\def\Cs{\operatorname{Cs}}

\def\Res{\operatorname{Res}}
\DeclareMathOperator{\divv}{div}
\def\numer{\operatorname{num}}
\newcommand{\R}{\mathbb{R}}

\DeclareMathOperator{\Div}{div}

\date{}
\dedicatory{} \commby{}

\author{J. D. Garc\'{\i}a-Salda\~{n}a}
\address{Departament de Matem\`{a}tiques \\
Universitat Aut\`{o}noma de Barcelona \\ Edifici C 08193 Bellaterra,
Barcelona. Spain} \email{johanna@mat.uab.cat}

\author{A. Gasull}
\address{Departament de Matem\`{a}tiques \\
Universitat Aut\`{o}noma de Barcelona \\ Edifici C 08193 Bellaterra,
Barcelona. Spain} \email{gasull@mat.uab.cat}

\author{H. Giacomini}
\address{Laboratoire de Math\'{e}matiques et Physique Th\'{e}orique.
 Facult\'{e} des
Sciences et Techniques. Universit\'{e} de Tours, C.N.R.S. UMR 7350.
37200 Tours. France} \email{Hector.Giacomini@lmpt.univ-tours.fr}

\subjclass[2000]{Primary 34C07, Secondary: 34C23, 34C25, 34C37, 37C27, 37C29,
49J15}

\keywords{Planar polynomial  system,
 uniqueness and hyperbolicity of the limit cycle, polycycle,  bifurcation,
 phase portrait on the Poincar\'{e} disc, Dulac function, stability, nilpotent point,
 basin of attraction}

\begin{document}

\title[Bifurcation diagram and stability for a one-parameter family]
{Bifurcation diagram  and stability for\\ a one-parameter family of
planar vector fields}

\begin{abstract}
We consider the 1-parameter family of  planar quintic systems, $\dot
x= y^3-x^3$, $\dot y= -x+my^5$, introduced by A. Bacciotti in 1985.
It is known that it has at most one limit cycle and that it can
exist only when the parameter $m$ is in $(0.36,0.6)$. In this paper,
using the Bendixon-Dulac theorem, we give a new unified proof of all
the previous results,  we shrink this to $(0.547,0.6)$, and we prove
the hyperbolicity of the limit cycle. We also consider the question
of the existence of polycycles. The main interest and difficulty for
studying this family is that it is not a semi-complete family of
rotated vector fields. When the system has a limit cycle, we also
determine explicit  lower bounds of the basin of attraction of the
origin. Finally we answer an open question about the change of
stability of the origin for an extension of the above systems.
\end{abstract}

\maketitle

\section{Introduction and main results}

A. Bacciotti, during a conference about the stability of analytic
dynamical systems, held in Florence in  1985, proposed to study the
stability of the origin of the following  quintic system
\begin{equation}\label{sism}
\left\{\begin{array}{lll}
\dot{x}=y^3-x^3,\\
\dot{y}=-x+my^5, \qquad m\in\mathbb{R}.
\end{array}\right.
\end{equation}

 Two years later, a quite complete study of
\eqref{sism} was done  by Galeotti and Gori in~\cite{Ga-Go}. They
prove that, when $m\in(-\infty, 0.36]\cup[0.6,\infty)$, system
\eqref{sism} has no limit cycles and, otherwise, it has at most one.
Their proofs are mainly based on the study of the stability of the
limit cycles, controlled by the sign of its characteristic exponent,
together with a transformation of the system using a special type of
adapted polar coordinates. Their proof of the uniqueness of the
limit cycle does not provide its hyperbolicity.

In this paper we  refine  the above results. To guess which is the actual
bifurcation diagram we did first a numerical study, obtaining the following: it
seems that there exists a value $m^*>0$, such that:

\begin{enumerate}[(i)]
\item  System \eqref{sism} has no limit cycles if
$m\in(- \infty,m^*]\cup[0.6,+\infty)$. Moreover, for $m=m^*$ it has
a heteroclinic polycycle  formed by the separatrices of the two
saddle points located at $(\pm m^{-1/4},\pm m^{-1/4})$.

\item For $m\in(m^*,0.6)$ the system has exactly one unstable limit cycle.

\item The value $m^*$ is approximately $0.560115.$

\end{enumerate}
Recall that a polycycle is a simple closed curve formed by several solutions of the
system and admitting  a Poincar\'{e} return map. The first two items coincide with the
ones described in~\cite{Ga-Go}. In  that paper it is claimed that $m^*$ is between
$0.58$ and $0.59$, but our computations give a different result, which  we believe
that is the right one.

The first aim of this work is to obtain analytic results that confirm, as much as
possible, the above description. To clarify the phase portraits of the system, we
will draw them on the Poincar\'{e} disc, see \cite{ALGM, So}.

For $m\leq 0$, system \eqref{sism} has no periodic orbits because
 $x^2/2+y^4/4$ is a global Lyapunov function. Therefore,
 the origin is a global attractor. In particular, its phase portrait  is trivial.
Therefore, we will concentrate on the case $m>0.$ In this case, the
system has three critical points,
 $(\pm m^{-1/4},\pm m^{-1/4})$ and $(0,0)$. The first couple of points  are saddles
 and the third one is a monodromic nilpotent singularity. Its stability can be
 determined using the tools introduced in \cite{AlGa,Mo}, see Subsection
 \ref{ss:monod} and next Theorem~\ref{mteo3}. We prove:

\begin{theorem}\label{mteo} Consider system~\eqref{sism}.
\begin{itemize}
\item[(i)]  It has neither periodic orbits, nor
polycycles, when $m\in(-\infty,0.547]\cup[0.6,\infty)$. Otherwise,
it has at most one periodic orbit or one polycycle and both can not
coexist. Moreover, when the limit cycle exists, it is hyperbolic and
unstable.

\item[(ii)] For $m>0,$ their phase portraits   on the
Poincar\'{e} disc, are given in Figure \ref{Sphere}.

\item[(iii)] Let $\mathcal M$ be the set of values of $m$ for which it has a
 heteroclinic polycycle. Then $\mathcal M$ is finite, non-empty and it is contained
 in $(0.547,0.6)$. Moreover, when $m\in{\mathcal M}$, the corresponding system has no
 limit cycles and its phase portrait  is given by Figure~\ref{Sphere} {\rm (b)}.
\end{itemize}
\end{theorem}

\begin{figure}[h]
\begin{tabular}{cccc}
\epsfig{file=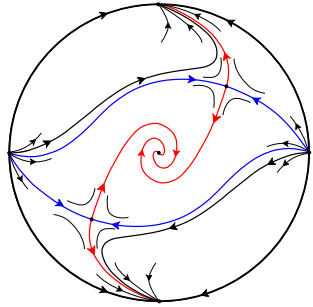,width=4cm, height=4cm}\,\, &
\epsfig{file=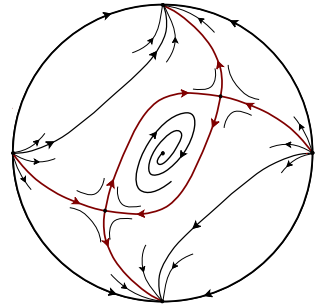,width=4cm, height=4cm}&\\
(a) When $m\in(0,0.547]$, or when  & (b) When $m\in(0.547,0.6)$ and \\
$m\in(0.547,0.6)$ and neither the &  the polycycle exists.\\polycycle nor the limit
cycle
exist.\\\\
\epsfig{file=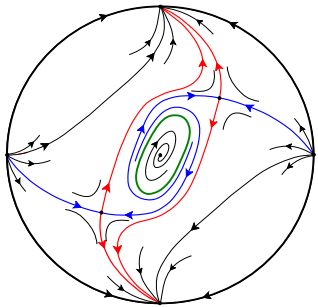,width=4.5cm, height=4cm}&
\epsfig{file=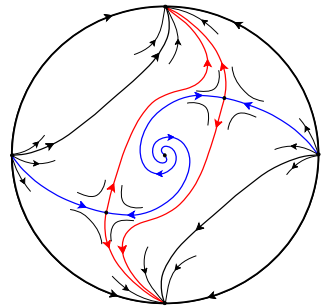,width=4.5cm, height=4cm}\\
(c) When $m\in(0.547,0.6)$ and & (d) For $m\in[0.6,\infty)$\\ the
limit cycle exists.
\end{tabular}
\caption{Phase portraits of system \eqref{sism}.}\label{Sphere}
\end{figure}

Our  simulations show that (a), (b) and (c) of Figure \ref{Sphere}
occur when $m\in(0,m^*)$, $m=m^*$ and $m>m^*$, respectively, for
some $m^*\in(0.547,0.6)$, that numerically we have found to be
$m^*\approx 0.560115.$ We have not been able to prove the existence
of this special value $m^*,$ because our  system  is not a
semi-complete family of rotated vector fields (SCFRVF) and this fact
hinders  the obtention of the full bifurcation diagram; see the
discussion in Subsection \ref{ss:nscf} and Example \ref{exam}. This
is precisely the reason for which we have decided to push forward in
the study of system~\eqref{sism}. Our approach can be useful to
understand other interesting polynomial systems of differential
equations that have been considered previously; see for instance
\cite{ca,Gu}.

In any case, from our result, we know the existence of finitely many
values $m^*_j,$ $j=1,\ldots,k,$ where $k\ge1$, satisfying
$0.547<m^*_1<m^*_2<\cdots m^*_k<0.6$, such that phase portrait~(b)
only happens for these values. Moreover, for $m\in (0.547,m^*_1)$,
phase portrait~(a) holds, for $m\in (m^*_k,0.6)$ phase portrait~(c)
holds, and for each one of the remainder $k-1$ intervals, the phase
portrait does not vary on each interval and is either~(a) or~(c).

As a byproduct of our approach we can also give explicit  algebraic restrictions on
the initial conditions to ensure that the solutions starting at them tend to the
origin.

Recall that when a critical point, $\mathbf{p}\in\R^n$,  of a
differential system is an attractor we can define its basin of
attraction as
\[
\mathcal{W}^s_{\mathbf{p}}=\{\mathbf{x}\in\R^n\,:\, \lim_{t\to+\infty}
\varphi(t,\mathbf{x})=\mathbf{p}\},
\]
where $\varphi$ denotes the solution of the differential system such
that $\varphi(0,\mathbf{x})=\mathbf{x}.$ A very interesting
question, mainly motivated by Control Theory problems, consists in
obtaining testable conditions for ensuring that some initial
condition is in $\mathcal{W}^s_{\mathbf{p}}.$ Usually these
conditions are obtained using suitable Lyapunov functions. We prove
next result using a different approach based on the construction of
Dulac functions.

\begin{proposition}\label{mteo2} Let $\mathcal{W}^s_{\mathbf{0}}$ be the
basin of attraction of the origin of system \eqref{sism}. Consider
$V_m(x,y)=g_{0,m}(y)+g_{1,m}(y)x+g_{2,m}(y)x^2$, with
\begin{align*}
g_{2,m}(y)=&\frac{1}{89100}(3-10m)(3+35m)y^{12}-\frac{1}{6300}(75-125m)^{2/3}(3-13m)y^8
\\&+\frac{1}{90}(3-10m)y^6-\frac{1}{25}(75-125m)^{2/3}y^2+1,
\end{align*}
$g_{1,m}(y)=g_{2,m}'(y)$ and
$g_{0,m}(y)=g_{2,m}''(y)/2-my^5g_{2,m}'(y)/2+5my^4g_{2,m}(y)/3.$
Then, for $m\in(0.5,0.6),$ $\mathcal{U}_m\subset
\mathcal{W}^s_{\mathbf{0}}$, where $\mathcal{U}_m$ is the bounded
connected component of $\{(x,y)\in\R^2\,:\, V_m(x,y)\le0\},$ that
contains the origin and whose boundary is the oval of $V_m(x,y)=0,$
see Figure~\ref{cuenca}.
\end{proposition}
\begin{figure}[h]
\epsfig{file=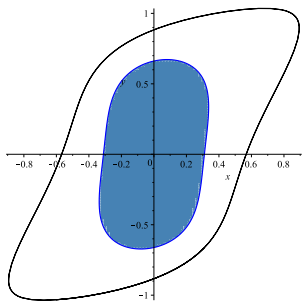,width=4cm, height=4cm} \caption{The limit
cycle of system~\eqref{sism} and the  set $\mathcal{U}_m$,
introduced in Proposition~\ref{mteo2},  when
$m=0.57.$}\label{cuenca}
\end{figure}

As we will see,  the proof of the above proposition is an
straightforward consequence of Proposition~\ref{uniq}. Using the
same tools,  it can be shown that the same result also holds for
smaller values of $m$. In any case, notice that this proposition
covers all the values of $m$ for which the system has limit cycles.

Studying the stability of the origin of system~\eqref{sism} we realized that, using
the same tools, we could solve an open question left in~\cite{Ga-Go}. Our third
result studies the stability of the origin of the following generalization of
system~\eqref{sism}:
\begin{equation}\label{sismg}
\left\{\begin{array}{lll}
\dot{x}=y^3-x^{2k+1},\\
\dot{y}=-x+my^{2s+1}, \qquad m\in\mathbb{R}\quad\mbox{and}\quad k,s\in\mathbb{N}^+.
\end{array}\right.
\end{equation}
In~\cite{Ga-Go}, the authors  gave the stability of the origin when
$s\ne2k$ and ask whether it is true or not that the change of
stability of the origin when $s=2k$ is at the value
$m=(2k+1)/(4k+1)$.  We will prove that their guess was not correct
for $k>1$. Next result shows that when $s=2k$, the stability changes
at
\begin{equation}\label{mmm}
 m=\frac{(2k+1)!!}{(4k+1)!!!!},
\end{equation}
where, given $n\in\mathbb{N}^+$,  $n!!$ and  $n!!!!$ are defined  recurrently, as
follows,
\[
n!!= n\times (n-2)!!, \quad  n!!!!= n\times (n-4)!!!!,
\]
with  $1!!=1, 2!!=2$ and $j!!!!=j$ for $1\le j\le 4.$ Notice that when $k=1,$ the
right hand-side of \eqref{mmm} and $(2k+1)/(4k+1)$  coincide and give  $m=3/5,$
which is one of the values appearing in Theorem~\ref{mteo}.
\begin{theorem}\label{mteo3} Consider system~\eqref{sismg}. Then:
\begin{itemize}
\item[(i)] When $s<2k$, the origin is an attractor when $m<0$ and a repeller when $m>0.$

\item[(ii)] When $s>2k$, the origin is always an attractor.

\item[(iii)] When $s=2k$, the origin is an attractor when $m<{(2k+1)!!}/{(4k+1)!!!!}$ and
a repeller when the reverse inequality holds. Moreover, when $k=1$ and $m=3/5$ the
origin is a repeller and for $m\lesssim 3/5$ system \eqref{sism} has at least one
limit cycle near the origin.
\end{itemize}

\end{theorem}

The  method used to study the stability of the origin of~\eqref{sismg}, when $s=2k$
and $k=1$, also works for deciding its stability  for the cases not covered by the
above theorem: $s<2k$ and $m=0$; and $s=2k$, $k>1$ and $m$ as in~\eqref{mmm}.
Nevertheless, the computations are tedious and we have decided do not perform them.

The paper is structured as follows. In Section \ref{ss:monod}  we prove Theorem
\ref{mteo3}. Section \ref{se:pr}  collects some preliminary results. It starts with
a discussion on the
 differences between being or not, a  SCFRVF.
 Then, subsection \ref{ss:retrato} is devoted to study
  the singularities at infinite of system \eqref{sism}
 and  their phase portraits on the Poincar\'{e} disc. Afterwards, we
 present some Bendixson-Dulac type results that we will use to prove
 non-existence or uniqueness of periodic orbits or polycycles.
 Finally, we introduce a result for controlling the number of roots  of 1-parameter  families of
 polynomials and we show that our system can be reduced to an Abel
 differential equation.

In Section \ref{se:ne1} we prove the non-existence results for $m\in(-\infty,
0.36]\cup[0.6,\infty)$. Our proof is different of that of \cite{Ga-Go} and it is
mainly based on the use of Dulac functions.

In Section \ref{se:uni} we prove the existence of at most one periodic orbit when
$m\in(1/2,0.6)$. Our approach also gives the hyperbolicity of the orbit and again
uses a Bendixson-Dulac type result. This section also includes the proof of
Proposition~\ref{mteo2}.

Section \ref{se:ne2} is devoted to enlarge the region where we can assure the  non
existence of periodic orbits and polycycles, proving this for $m\in(9/25,0.547]$.
The proof uses once more a suitable Dulac function in a part of the interval and
the Poincar\'{e}-Bendixon theorem, together with the hyperbolicity of the limit cycle,
whenever it exists, for the remaining values of $m.$

Section \ref{se:hetero} deals with the existence of polycycles for the system.
Finally in Section \ref{se:teo} we glue all
the above results to prove Theorem \ref{mteo}.

\section{Stability of the origin and proof of Theorem \ref{mteo3}}\label{ss:monod}

Notice that the origin of \eqref{sism} and \eqref{sismg} are nilpotent critical
points and there are several tools for studying its local stability, see for
instance \cite{AlGa,Ga-To,Mo}. We will follow the approach of \cite{AlGa,Ga-To},
based on the polar coordinates introduced by Lyapunov in~\cite{lia}, to study of
the stability of degenerate critical points.

Let $u(\varphi)=\Cs(\varphi)$ and $v(\varphi)=\Sn(\varphi)$ be the solutions of the
Cauchy problem:
\[
\dot u=-v^{2p-1},\,\dot v =u^{2q-1},\quad u(0)=\root{2q}\of{1/p}\quad \mbox{ and
}\quad v(0)=0,
\]
where the prime denotes the derivative with respect to $\varphi$.

The Lyapunov \textit{generalized polar coordinates}  are $x=r^p\Cs(\varphi)$ and
$y=r^q\Sn(\varphi)$. They  parameterize the algebraic curves $p x^{2q}+q
y^{2p}=r^{2pq},$ that correspond to the level sets of above
$(p,q)-$quasi-homogeneous Hamiltonian system. In particular,
$p\Cs^{2q}(\varphi)+q\Sn^{2p}(\varphi)=1$, and both functions are smooth
$T_{p,q}$-periodic functions, where
$$
T=T_{p,q}=2p^{-1/2q}q^{-1/2p}
\frac{\Gamma\left(\frac{1}{2p}\right)\Gamma\left(\frac{1}{2q}\right)}
{\Gamma\left(\frac{1}{2p}+\frac{1}{2q}\right)},
$$
and $\Gamma$ denotes the  Gamma function. The general expression of
a differential system in these coordinates is:
\begin{equation}\label{cam}
\dot{r}=\frac{x^{2q-1}\dot{x}+y^{2p-1}\dot{y}}{r^{2pq-1}},\quad
\dot{\theta}=\frac{px\dot{y}-qy\dot{x}}{r^{p+q}}.
\end{equation}

In the nilpotent monodromic  case, the component $\dot{\theta}$ does not vanish in
a punctured neighborhood of the critical point. Hence, system \eqref{cam} can be
written in a neighborhood of $r=0$  as
\begin{equation}\label{eqR}
\frac{dr}{d\theta}=\sum_{i=1}^{\infty}R_i(\theta)r^i,
\end{equation}
where $R_i(\theta)$, $i\geq1$ are $T$-periodic functions. The solution of
\eqref{eqR} that for $\theta=0$ passes for $r=\rho$ can be written as the power
series
\begin{equation}\label{rrr}
r(\theta,\rho)=\rho+\sum_{i=2}^{\infty}u_i(\theta)\rho^i,\quad \mbox{with}\quad
u_i(0)=0,
\end{equation} and the functions $u_i$ can be computed solving recursive
linear differential equations obtained plugging~\eqref{rrr} in~\eqref{eqR}. It is
well-known that the stability of the origin is given by the first non-vanishing
generalized Lyapunov constant $V_k:=u_k(T)$.

To effectively  compute some integrals of the above generalized trigonometric
functions we will use the following result, see~\cite{Ga-To}.
\begin{lemma}\label{calcul} Let $\Sn$ and $\Cs$ be the (1,q)-trigonometrical functions and
let $T$ be their period. Then, for $i,j\in\mathbb{N}$,
\begin{itemize}
\item[(i)] $\int_0^T\Sn^i(\theta)\Cs^j(\theta)\,d\theta=0$ when
either $i$ or $j$ are odd.
\item[(ii)] $\int_0^T\Sn^i(\theta)\Cs^j(\theta)\,d\theta=
\dfrac{2\Gamma\left(\frac{i+1}{2}\right)\Gamma\left(\frac{j+1}{2q}\right)}
{q^{\frac{i+1}{2}}\Gamma\left(\frac{i+1}{2}+\frac{j+1}{2q}\right)}$ when $i$ and
$j$ are both even.
\item[(iii)]  For $q=2,$ $\int_0^\theta
\Cs^8(\psi)\,d\psi=\dfrac{6\Sn(\theta)\Cs^5(\theta)+10\Sn(\theta)\Cs(\theta)+5\theta}{21}.$
\item[(iv)]  For $q=2$, $\int_0^\theta
\Sn^4(\psi)\,d\psi=\dfrac{-\Sn^3(\theta)\Cs(\theta)-\Sn(\theta)\Cs(\theta)+\theta}{7}.$
\end{itemize}
\end{lemma}

\begin{proof}[Proof of Theorem~\ref{mteo3}]
By using the transformation $(x,y)\to (y,x)$, system \eqref{sismg} becomes
\begin{equation}\label{sismt}
\left\{\begin{array}{lll} \dot{x}=-y+mx^{2s+1},\\
\dot{y}=x^3-y^{2k+1}.
\end{array}\right.
\end{equation}
We use \eqref{cam}, with $p=1$ and $q=2$, to transform it into
\begin{equation*}\label{PoGral}
\left\{\begin{array}{lll}
\dot{r}=m\Cs^{2s+4}(\theta)r^{2s+1}-\Sn^{2k+2}(\theta)r^{4k+1},
\\
\dot{\theta}=r-\Cs(\theta)\Sn^{2k+1}(\theta)r^{4k}-2m\Cs^{2s+1}(\theta)\Sn(\theta)r^{2s},
\end{array}\right.
\end{equation*}
or equivalently,
\begin{equation}\label{EqGe}
\frac{dr}{d\theta}=\frac{m\Cs^{2s+4}(\theta)r^{2s}-\Sn^{2k+2}(\theta)r^{4k}}{
1-\Cs(\theta)\Sn^{2k+1}(\theta)r^{4k-1}-2m\Cs^{2s+1}(\theta)\Sn(\theta)r^{2s-1}}.
\end{equation}
The Taylor series of the right hand-side of \eqref{EqGe} at the origin has three
possibilities:

(i) When $s<2k$, then \eqref{EqGe} becomes
\begin{equation*}
\frac{dr}{d\theta}=m\Cs^{2s+4}(\theta)r^{2s}+O(r^{4k}).
\end{equation*}
Therefore, using the method explained above and  Lemma~\ref{calcul}, we get that
its first Lyapunov constant is
\begin{equation}\label{s<2k}
V_{2s}=m\int_0^T
\Cs^{2s+4}(\theta)\,d\theta=\frac{m\,\sqrt{2\pi}\,\Gamma\left(\frac{2s+5}{4}\right)}
{\Gamma\left(\frac{2s+7}{4}\right)}.
\end{equation}
Then $m=0$ is the bifurcation value, and the origin of~\eqref{sismg} changes its
stability from attractor to repeller as $m$ goes from negative values to positive
values.

(ii) Suppose $s>2k$, then the Taylor expansion of \eqref{EqGe} at $r=0$ is
\begin{equation*}
\frac{dr}{d\theta}=-\Sn^{2k+2}(\theta)r^{4k}+O(r^{2s}).
\end{equation*}
By using the same method, we obtain that the first Lyapunov constant is
\begin{equation}\label{s>2k}
V_{4k}=\int_0^T
-\Sn^{2k+2}(\theta)\,d\theta=-\frac{\Gamma\left(\frac{1}{4}\right)\Gamma\left(\frac{2k+3}{2}\right)}
{2^{\frac{2k+1}{2}}\Gamma\left(\frac{4k+7}{4}\right)}<0,
\end{equation}
and the stability of the origin of~\eqref{sismg} is independent  of $m$ and it is
an attractor for all $m$.

(iii) Finally, when $s=2k$ we have
\begin{equation}\label{ser}
\frac{dr}{d\theta}=\left(m\Cs^{4k+4}(\theta)-\Sn^{2k+2}(\theta)\right)r^{4k}+O(r^{8k-1}).
\end{equation}
Hence the first non-vanishing generalized Lyapunov constant is given by
$$
V_{4k}=\int_0^T \left(m\Cs^{4k+4}(\theta)-\Sn^{2k+2}(\theta)\right)d\theta.
$$
By using \eqref{s<2k} with $s=2k$ and \eqref{s>2k}, after some simplifications,  we
obtain that
$$
V_{4k}=\frac{2\pi^{3/2}\big(m(4k+1)!!!!-(2k+1)!!\big)}
{\left(\Gamma\left(\frac{3}{4}\right)\right)^2(4k+3)!!!! }. $$ Therefore the origin
of~\eqref{sismg} is attractor for $m<(2k+1)!!/(4k+1)!!!!$ and repeller for
$m>(2k+1)!!/(4k+1)!!!!$, as we wanted to prove.

In the particular case $s=2k$ and $k=1$, that corresponds to
system~\eqref{sism}, and when $m=3/5$ we have that $V_4=0.$ To proof
the theorem we continue computing the next non-zero Lyapunov
constant. For $s=2$ and $k=1$, equation~\eqref{EqGe}  writes as
\begin{equation*}
\frac{dr}{d\theta}=R_4(\theta)r^4+R_7(\theta)r^7+R_{10}(\theta)r^{10}+O(r^{13}),
\end{equation*}
with $R_4(\theta)=m\Cs^8(\theta)-\Sn^4(\theta)$,
\begin{equation*}
R_7(\theta)=2m^2 \Cs^{13}(\theta)\Sn(\theta) +m \Cs^9(\theta)\Sn^3(\theta)-2m
 \Cs^5(\theta)\Sn^5(\theta)-\Cs(\theta)\Sn^7(\theta)
\end{equation*}
and
\begin{align*}
R_{10}(\theta)=&4 m^3\Cs^{18}(\theta)\Sn^2(\theta)+4
m^2\Cs^{14}(\theta)\Sn^4(\theta)+m (1-4 m) \Cs^{10}(\theta)\Sn^6(\theta)
\\&-4m\Cs^6(\theta)\Sn^8(\theta)-\Cs^2(\theta)\Sn^{10}(\theta),
\end{align*}
with $m=3/5$. Following the procedure explained in this section we
obtain that $u_2=u_3=0$,
\begin{align*}
u_4(\theta)&= \int_ 0^\theta R_4(\psi)\,d\psi, \quad u_5=u_6=0,\\
u_7(\theta)&= \int_0^\theta \big(R_7(\psi)+4R_4(\psi)u_4(\psi)\big)\,d\psi,
\quad u_8=u_9=0,\\
u_{10}(\theta)&= \int_0^\theta
\Big(R_{10}(\psi)+7R_7(\psi)u_4(\psi)+4R_4(\psi)u_7(\psi)+6R_4(\psi)u_4^2(\psi)
\Big)\,d\psi.
\end{align*}
Using Lemma~\ref{calcul} and some easy computations we get  that  $V_1= \cdots =V_9
= 0$. Finally, it suffices to compute
\[
V_{10}= \int_0^T
\Big(R_{10}(\theta)+7R_7(\theta)u_4(\theta)+4R_4(\theta)u_7(\theta)\Big)\,d\theta,
\]
because  $\dfrac{du_4^3(\theta)}{d\theta}=3R_4(\theta)u_4^2(\theta).$ Using
integration by parts and the expression of $u_7'$ we arrive to
\begin{equation}\label{v10}
V_{10}= \int_0^T
\Big(R_{10}(\theta)+3u_4(\theta)u_7'(\theta)\Big)\,d\theta=\int_0^T
\Big(R_{10}(\theta)+3u_4(\theta)R_7(\theta)\Big)\,d\theta.
\end{equation}
Notice that applying (iii) and (iv) of Lemma~\ref{calcul} we know that
\[
u_4(\theta)=\int_0^\theta
\Big(\frac35\Cs^8(\psi)-\Sn^4(\psi)\Big)\,d\psi= \frac{
6\Sn(\theta)\Cs^5(\theta)+15\Sn(\theta)\Cs(\theta)+
5\Sn^3(\theta)\Cs(\theta)}{35}.
\]
Plugging this expression in~\eqref{v10}, using several times (i) and (ii) of
Lemma~\ref{calcul} and the properties of the $\Gamma$ function we arrive to
\begin{align*}
V_{10}=\frac{128}{1625}
\frac{\left(\Gamma\left(\frac{3}{4}\right)\right)^2}{\sqrt{\pi}}>0.
\end{align*}
Hence the origin is  unstable for $m= 3/5$. As a consequence, we know that  at $m=
3/5$ the system has a Hopf-like bifurcation. Therefore the system  has at least one
limit cycle near the origin for $m\lesssim 3/5$.
\end{proof}

\section{More preliminary results}\label{se:pr}

This section is a miscellaneous one and it is divided into several short
subsections containing either some tools that we will use to prove Theorems
\ref{mteo} and \ref{mteo2} or some preliminary results.

\subsection{Differences between  families that are SCFRVF and families
that are not}\label{ss:nscf}

On the one hand, if a one-parameter family of differential systems
is a SCFRVF, then there are many results that allow to control the
possible bifurcations; see \cite{Duff, Perko, Perko1}. One of the
most useful ones is the so called {\it non-intersection property.}
It asserts that if $\gamma_{1}$ and $\gamma_{2}$ are limit cycles
corresponding to systems with different values of $m,$ then
$\gamma_{1}\cap \gamma_{2}=\emptyset.$ Informally, we like to call
this property {\it Atila's property}\footnote{Recall that about
Atila, King of the Huns,  it was said that ``the grass never grew on
the spot where his horse had trod".}
 because it implies that, if for some value of $m$ a
limit cycle passes trough a region of the phase plane, this region
turns out to be forbidden for the periodic orbits that the system
could have for
 any other value of the parameter. As a consequence, in this case, the study
of  1-parameter bifurcation diagrams is much simpler.

For instance, consider a  1-parameter SCFRVF satisfying the
following property:

{\noindent \bf (P)} {\it For each $m\in(m_0,m_1)$, it has at most
one limit cycle, that we denote by $\gamma_m$. Here, if for some $m$
the corresponding system has no limit cycles then
$\gamma_m=\emptyset$.  Moreover, assume that
$\cup_{m\in(m_0,m_1)}\gamma_m$ covers a region of the plane where
all  the periodic orbits of the system have to cut. }

Therefore,  it holds that: {\it
 for $m\in \R\setminus(m_0,m_1)$ the system has no periodic
orbits.}

The above property has very important practical consequences if we want to
determine the values $m_0$ and $m_1$, that constitute, in many cases, the most
difficult ones to be obtained to complete the bifurcation diagram. Usually, one of
the values, say $m_0$ corresponds to a Hopf-like bifurcation, and some local
analysis allows to obtain it. Then, for instance,  if   for some value of $m$, say
$\widetilde m>m_0$, the system has no limit cycles then $m_1<\widetilde m$.  The
same idea can also be applied to obtain lower bounds of $m_1$. These facts simplify
a lot the obtention of analytic bounds for the value $m_1$ because it suffices to
deal with concrete systems, with fixed values of $m$. This approach has been
applied with success in many works; see for instance \cite{Ga-Ga-Gi,Ga-Gi-To1,Pe,
Perko1,Xian}.

On the other hand, if for a general family of vector fields, we have
that the same property {\bf (P)} given above holds, we can  say
nothing of what happens for $m\in \R\setminus(m_0,m_1)$. For this
reason, when we  study system \eqref{sism}, we can not ensure the
existence of a unique value  of $m$ for which phase portrait (b) of
Figure \ref{Sphere} appears; see also Example~\ref{exam}. We remark
that system \eqref{sism} is not a SCFRVF with respect to $m$, and
moreover we have not been able to transform it into an equivalent
one that were a SCFRVF.

 From our point of view, to introduce tools for studying  1-parameter families that
 are not SCFRVF is a challenge for the differential equations
 community.

\subsection{Global phase portrait}\label{ss:retrato}

We will draw the phase portraits of system \eqref{sism} on the
Poincar\'{e} disc, \cite{ALGM, So}. Recall that, from  the works of
Markus \cite{Ma} and Newmann \cite{Ne}, for knowing a phase portrait
it suffices to determine the type of critical points (finite and at
infinity), the configuration of their separatrices, and the number
and character of their periodic orbits.

We start  making a study of the critical points at infinity of the Poincar\'e
compactification of the system.  That is, we will use the transformations
$(x,y)=(1/z, u/z)$ and $(x,y)= (v/z, 1/z)$, with a suitable change of time to
transform system \eqref{sism} into two new polynomial systems, one in the $(u,
z)$-plane and another one in the $(v,z)$-plane respectively; see \cite{ALGM} for
the details. Then, for understanding the behavior of the solutions of~\eqref{sism}
near infinity it suffices to study  the type of critical points of the transformed
systems which are localized on the line $z=0$. These points are precisely the so
called critical points at infinity of system \eqref{sism}.

\begin{lemma}
By using the transformation $(x,y)=\left(v/z,1/z\right)$ and the
change of time $dt/d{\tau}=1/{z^4}$ system \eqref{sism} is
transformed into the system
\begin{equation}\label{CartaVZ}
\left\{\begin{array}{l}
v'=-mv+(1-v^3)z^2+v^2z^4,\\
z'=-mz+vz^5,
\end{array}\right.
\end{equation}
where the prime denotes the derivative with respect to $\tau$. The
origin is the unique critical point of \eqref{CartaVZ} on $z=0$ and
it is an attracting node.
\end{lemma}

The proof of the above result is straightforward.

\begin{lemma}
By using the transformation $(x,y)=\left(1/z,u/z\right)$ and the
change of time $dt/d{\tau}=1/{z^4}$ system \eqref{sism} is
transformed into the system
\begin{equation}\label{CartaUZ}
\left\{\begin{array}{l}
u'=(u-z^2)z^2+u^4(mu-z^2),\\
z'=(1-u^3)z^3,
\end{array}\right.
\end{equation}
where the prime denotes the derivative with respect to $\tau$. The
origin is the unique critical point of \eqref{CartaUZ} on $z=0$ and
it is a repeller.
\end{lemma}
\begin{proof}

From the expression of \eqref{CartaUZ} it is clear that the origin
is its unique critical point on $z=0$. For determining its nature we
will use the directional blow-up since the linear part of the system
at this point vanishes identically; see again \cite{ALGM}.

We apply the $z$-directional blow-up given by the transformation
$r={u}/{z}$, $z=z$. Performing it, together with the change of time
$dt/d\tau ={z^3}$, system \eqref{CartaUZ} is transformed into
\begin{equation}
\left\{\begin{array}{l}\label{sisrz}
\dot r=-1+mzr^5,\\
\dot z=1-z^3r^3.
\end{array}\right.
\end{equation}
System \eqref{sisrz} has no critical points on $z=0$. Then by using the
transformation $(u,z)=(rz,z)$ we can obtain the phase portrait of system
\eqref{sisrz}. Recall that the mapping swaps the third and fourth quadrants in the
$z$-directional blow-up. In addition, taking into account the change of time
$dt/d\tau =z^3$, it follows that the vector field in the third and fourth quadrant
of the plane $(u,z)$ has the opposite direction to the one obtained in the
$(r,z)$-plane.

Next, we need to perform the $u$-directional blow-up for knowing the
phase portrait in such  direction. After that, by joining the
information about the blow-ups in both directions, we will have the
phase portrait of system \eqref{CartaUZ}.

The $u$-directional blow-up is given by the transformation $u=u$,
$q={z}/{u}$, and with the change of time $dt/d\tau ={u^3}$, system
\eqref{CartaUZ} is transformed into
\begin{equation}
\left\{\begin{array}{l}\label{sisuq}
\dot u=-q^2(uq^2-1)-u^2(uq^2-m),\\
\dot q=q^5-muq.
\end{array}\right.
\end{equation}
On $q=0$,  the origin is the unique critical point of the system,
and since the linear part of the system at this point vanishes
identically we have to use again some directional blow-ups.

Since the lower degree term of $\dot qu-\dot u q$ is
$-q(2mu^2+q^2)$, and it only vanishes on the direction $q=0$,  to
study the origin of system \eqref{sisuq} it suffices to consider the
$u$-directional blow-up. It is given by the transformation $u=u$,
$s={q}/{u}$. Doing the change of time $dt/d\tau={u},$ system
\eqref{sisuq} becomes
\begin{equation}\label{sissu}
\left\{\begin{array}{l}
\dot u=-us^2(u^3s^2-1)-u(u^3s^2-m),\\
\dot s=s^3(u^3-1)+2s(u^3s^4-m).
\end{array}\right.
\end{equation}
For $s=0$, system \eqref{sissu} has a unique critical point at the
origin. The linearization matrix at the origin has eigenvalues $m$
and $-2m$. Thus the origin of system \eqref{sissu} is a saddle.

Then by using the transformation $(u,q)=(u,su)$ we can obtain the
phase portrait of system \eqref{sisuq}. Recall that the mapping
swaps the second and the third quadrants in the $u$-directional
blow-up. In addition, taking into account the change of time
$dt/d\tau =u$ it follows that the vector field in the second and
third quadrants of the plane $(u,q)$ has the opposite direction to
the one in the $(u,s)$-plane. Once we have the phase portrait in the
$(u,q)$-plane, we apply the transformation $(u,z)=(u,qu)$.

By considering the properties of the blow-up technic and from the
analysis of all the intermediate phase portraits we obtain that the
origin of system \eqref{CartaUZ} is a repeller.
\end{proof}

Recall that the finite critical points are two hyperbolic saddles at
$(\pm m^{-1/4},\pm m^{-1/4})$ and  a monodromic nilpotent
singularity $(0,0)$, whose stability is given in
Theorem~\ref{mteo3}. Finally notice that the vector field is
symmetric with respect to the origin.  By joining to these
properties  all the information concerning  the infinite critical
points and using the existence and uniqueness results on the number
of limit cycles and polycycles given in Theorem \ref{mteo},  we
obtain the global phase portraits of system \eqref{sism} given in
Figure \ref{Sphere}.

\subsection{Some  Bendixson-Dulac type  criteria}\label{ss:dulac}

Next statement is a  Bendixson-Dulac type result, that mixes the
Bendixson-Dulac Test  given in the classical book \cite[Thm.
31]{ALGM} and the  one given in \cite[Prop. 2.2]{Ga-Gi-2}. It is
adapted to our interests. Similar results appear in
\cite{che,Ga-Gi,Llo,Ya}.

\begin{proposition}[Bendixson-Dulac Criterion]\label{GenBenDul}
Let $X=(P,Q)$ be the vector field associated to the
$C^1$-differential system
\begin{equation}\label{sisX}
\left\{\begin{array}{lll}
\dot{x}=P(x,y),\\
\dot{y}=Q(x,y),
\end{array}\right.
\end{equation}
and let $\mathcal{U}\subset\R^2$ be an open region  with boundary
formed by finitely many algebraic curves. Assume that there exists a
rational function $V(x,y)$ and $k\in \R^+$ such that
\begin{equation}\label{M}
M= M_{\{V,k\}}(x,y)=\langle\nabla V,X\rangle-k V \Div(X)
\end{equation}
does not change sign in $\mathcal{U}$ and $M$ only vanishes on
points, or curves that are not invariant by the flow of $X$. Then,

\begin{itemize}
\item[(I)] If all the connected components of
 $\mathcal{U}\setminus\{V=0\}$ are
 simple connected then the system  has neither periodic orbits nor
 polycycles.

\item[(II)] If all the connected components of
 $\mathcal{U}\setminus\{V=0\}$ are
 simple connected, except  one, say \,$\widetilde{\mathcal U}$, that  is
1-connected, then, either the system  has neither periodic orbits nor polycycles or
it has at most one of them in $\mathcal{U}$. Moreover, when it has a limit cycle,
it is hyperbolic,  is contained in $\widetilde{\mathcal U}$, and its stability is
given by the sign of $-VM$ on~\,$\widetilde{\mathcal U}$.
\end{itemize}

\end{proposition}

\begin{proof}
Consider the Dulac function $g(x,y)=|V(x,y)|^{-1/k}$. Then
\begin{align*}
\divv(gX)&= \frac{\partial g}{\partial x}P+\frac{\partial
g}{\partial y}Q +g( \frac{\partial P}{\partial x}+ \frac{\partial
Q}{\partial y} )= \langle \nabla g,X \rangle
+g\divv(X)\\
&= -\frac1k\operatorname{sign}(V)|V|^{-\frac{k+1}{k}}\left(\langle \nabla V,X
\rangle -k
     V\divv (X)\right)\\
 &= -\frac1k\operatorname{sign}(V)|V|^{-\frac{k+1}{k}} M_{\{V,k\}}=
 -\frac1k\operatorname{sign}(V)|V|^{-\frac{k+1}{k}} M.
\end{align*}
By the  hypotheses,  $  M|_{\{V=0\}}=\langle \nabla V,X \rangle|_{\{V=0\}}$ does
not change sign in $\mathcal U$ and there is no solution contained in $\{M=0\}.$
Therefore, neither the  periodic orbits nor the polycycles of the vector field in
$\mathcal U$ can intersect $\{V=0\}.$

For  proving (I) we follow the proof  of the   Bendixson-Dulac Criterion given
in~\cite[Thm. 31]{ALGM}. Assume, to arrive a contradiction, that the system has a
simple closed curve $\Gamma$ which is union of trajectories of the vector field.
Let $C\subset \mathcal{U}$ the bounded region with boundary $\Gamma.$ Then, by
Stokes Theorem, we have that
$$ \iint_C\divv(gX)=\int_{\Gamma}\langle gX,\bold n\rangle,$$
where $\Gamma$ is oriented in the suitable way.
 Note that the right hand-side term in this equality is
zero because $gX$ is tangent to the curve $\Gamma$ and  the left one is non-zero by
our hypothesis. This fact leads to the desired contradiction.

In case (II), applying  a similar argument to the region bounded by   two possible
simple closed curves formed by trajectories of the vector field, we arrive again to
a contradiction.

To end the proof, let us show the hyperbolicity of the possible
limit cycle~$\Gamma$. Write $\Gamma=\{\gamma(t):=(x(t),y(t)),
t\in[0,T]\}\subset \widetilde U,$ where $T$ is its period, and its
characteristic exponent as $h(\Gamma)=\int_0^T \divv\left(
X(\gamma(t))\right)\,dt.$ We need to prove that $h(\Gamma)\ne0$ and
that its sign coincides with the sign of $-VM$ on $\widetilde U.$ We
know that
$$ \frac{M}{V}= \frac{\langle \nabla V,X \rangle}{V}-k \divv(X).$$
Remember that $\Gamma\cap\{V=0\}=\emptyset.$ Evaluating this last equality on
$\gamma$ and integrating between 0 and $T$ we obtain that
\begin{align}\label{bd}
\int_0^T \frac{M}{V}(\gamma(t) )\,dt =& \int_0^T \frac{\langle \nabla V,X
\rangle}{V}(\gamma(t))  \,dt-k \int_0^T\divv(X)(\gamma(t))\,dt\nonumber\\=& \ln|
V(\gamma(t))|\Big|_{t=0}^{t=T} -k \,h(\Gamma)= -k\, h(\Gamma).
\end{align}
Therefore, the result follows.
\end{proof}

Next result is an straightforward consequence of the above proposition. It notices
that when we construct a suitable Dulac function, the same method provides an
effective estimation of the basin of attraction of the attracting critical points.

\begin{corollary}\label{co:basin} Assume that we are under the hypotheses
of the above theorem and moreover that $\{V(x,y)=0\}$ has an oval
such that it and the bounded region surrounded by it, say $\mathcal
W$,  are contained in $\mathcal U$. Then, if the differential system
has only a critical point ${\bf p}$ in $\mathcal{W}$, and it is an
attractor,  then $\mathcal W$ is contained in the basin of
attraction of~$\bf p$.
\end{corollary}

Observe that when we are under the hypotheses of the above
corollary, but we already know that the system has a limit cycle in
$\mathcal U$, and $\mathcal U$ is simply connected,  then, unless
the set $\{V(x,y)=0\}$ reduces to a single point, there is no need
to assume that $\{V(x,y)=0\}$ has  an oval. The existence of the
oval is already guaranteed by the method itself.

Sometimes the hypothesis that $M$ does not change sign can be replaced for another
one, following next remark.

\begin{remark}\label{GenBenDul2}
Assume that in Proposition~\ref{GenBenDul} we can not ensure that
the function $M$, given in~\eqref{M}, keeps sign on the whole
domain~$\mathcal{U}$. Then, this hypothesis can be changed by
another one. Define  $\{M=0\}^*$ to be the subset of $\{M=0\}$
formed by curves that separate the regions $\{M>0\}$ and $\{M<0\}$.
Then, the new hypothesis is that the set $\{M=0\}^*$  is without
contact by the flow of $X$. Then, in the conclusions of the
proposition,  the connected components of
$\mathcal{U}\setminus\{V=0\}$ must be replaced by the connected
components of $\mathcal{U}\setminus\big(\{V=0\}\cup \{M=0\}^*\big)$
and the same type of conclusions hold. We will use this idea in the
proof of Proposition~\ref{925}.
\end{remark}

\subsection{Zeros of 1-parameter families of polynomials}\label{ss:polin}

 As usual, for a polynomial $P(x)=a_nx^n+\cdots+a_1x+a_0,$  we write $\triangle_x(P)$ to
denote its discriminant, that is,
$$
\triangle_x(P)=(-1)^{\frac{n(n-1)}{2}}\frac{1}{a_n}\operatorname{Res}(P(x),P'(x)),
$$
where $\operatorname{Res}(P,P',x)$ is the resultant of $P$ and $P'$
with respect to $x$; see \cite{cox}.

By using the same techniques that in \cite[Lem. 8.1]{Ga-Ga-Gi}, it is not difficult
to  prove the following result that will be used in several parts of the paper.
\begin{lemma}\label{ll}
Let $ G_m(x)=g_n(m)x^n+g_{n-1}(m)x^{n-1}+\cdots+g_1(m)x+g_0(m)$ be a
family of real polynomials depending continuously on a real
parameter~$m$ and set $\Lambda_m=(c(m),d(m))$ for some continuous
functions $c(m)$ and $d(m)$. Suppose that there exists an  interval
$I\subset\R$ such that:
\begin{itemize}
\item[(i)] For  some $m_0\in I$,  $G_{m_0}$ has exactly $r$  zeros  in $\Lambda_{m_0}$
and all them are simple.


\item[(ii)] For all $m\in I$, $G_m(c(m))\cdot G_m(d(m))\ne0$.

\item[(iii)] For all $m\in I,$ $\triangle_x(G_m)\neq 0.$

\end{itemize}
Then for all $m\in I$, $G_m(x)$ has also exactly $r$  zeros in $\Lambda_m$ and all
them are simple.

\end{lemma}
 The  idea of the proof consists in
looking at the roots of $G$ as continuous functions  of $m$. The
hypothesis (ii) prevents that some real roots of $G_m$ passes,
varying $m$, trough the boundary of $\Lambda_m$. The hypothesis
(iii) forbids, that varying $m$, appears some multiple root of
$G_m$.

Notice that the above result transforms the control of the zeros of
a function depending on two variables, $x$ and $m$, into three
problems of only one variable, the one of item (i) with the variable
$x$ and the two remainder ones with the variable $m$. If the
dependence on $m$ is also polynomial, and the polynomial has
rational coefficients, then these three simpler questions can be
solved by applying the well-known Sturm method. As we will see in
the proof of Proposition \ref{uniq}, this approach can also extended
when the one variable polynomial has some irrational coefficients.

\subsection{Transformation into an Abel equation}

System \eqref{sism} can be seen as the sum of two quasi-homogeneous
vector fields, see \cite{cgp}. It is known that in many cases these
systems can be transformed into Abel equations. We  get:

\begin{proposition}
The periodic orbits of system \eqref{sism} correspond to positive
$T$-periodic solutions of the Abel equation
\begin{equation}\label{abel}
\frac{d\rho}{d\theta}=\alpha(\theta)\rho^3+\beta(\theta)\rho^2,
\end{equation}
where
$$\alpha(\theta)=3\Cs(\theta)\Sn(\theta)\left(2m\Cs^{4}(\theta)+\Sn^2(\theta)\right)
\left(m\Cs^{8}(\theta)-\Sn^4(\theta)\right)
$$
and
$$
\beta(\theta)=5m\Cs^{8}(\theta)-4\Sn^4(\theta)+(3-10m)\Cs^{4}(\theta)\Sn^2(\theta),
$$
being $\Sn$ and $\Cs$  the functions introduced in Subsection
\ref{ss:monod} and $T$ their period.
\end{proposition}

\begin{proof} The result follows by applying the Cherkas
transformation
\[
\rho=\frac{r^3}{1-r^3\Sn(\theta)\Cs(\theta)\left(\Sn^2(\theta)+2m\Cs^4(\theta)\right)},
\]
to the expression of system \eqref{sism} in the quasi-homogeneous
polar coordinates introduced in Section \ref{ss:monod}. It is used
that the periodic orbits of the system do not intersect the curve
$\dot\theta=0$, and therefore the above transformation is
well-defined over them, see \cite{cgp}.
\end{proof}

Using the above expression it is not difficult to reproduce the
proof of the existence of the Hopf-like bifurcation given in
Subsection~\ref{ss:monod}. Unfortunately, although
expression~\eqref{abel} looks quite simple, the results about the
number of limit cycles of Abel equations that we know are not
applicable to \eqref{abel}.

\section{Non-existence of limit cycles for
$m\in(0,9/25)\cup(3/5,\infty)$}\label{se:ne1}

In this section we prove the non-existence results of periodic
orbits already given in \cite{Ga-Go} and extend them to the
non-existence of polycycles. Our proof is different and based on the
Bendixson-Dulac theorem and other classical tools. We study
separately each interval.

\begin{proposition}\label{nc} For $m\in(0,9/25]$,
system \eqref{sism}  has neither periodic orbits nor polycycles.
\end{proposition}
\begin{proof}
Recall that for $m\in(0,9/25]$ the origin is attractor. Therefore if
we prove that any periodic orbit $\Gamma$ of the system is also
attractor we will have proved that the system has  not periodic
orbits. In order to prove the stability of the limit cycle we need
to compute $\int_0^{T} \Div\left(X(\gamma(t))\right)\,dt$, where
$\gamma(t):=(x(t),y(t))$ is the time parametrization of $\Gamma$,
and $T=T(\Gamma)$ its period.

From equation \eqref{M}, for any function $V$ such that
$\{V(x,y)=0\}\cap \Gamma=\emptyset$, we have
$$
\Div(X)=\frac{M_{\{V,k\}}-\langle\nabla V,X\rangle}{-k V}.
$$
Hence,
$$
\int_0^{T}\Div\left(X(\gamma(t))\right)\,dt=-\int_0^{T}\frac{M_{\{V,k\}}(\gamma(t))}{k
V(\gamma(t))}dt,
$$
where we have followed similar computations that in~\eqref{bd}. Then the stability
of $\Gamma$ is given  by the sign of $-MV$. If we show that for $m\in(0,9/25]$
there exist a non-negative $V$ and $k\in\R^+$, such that
 its corresponding $M$ is non-negative, then we will have proven that the limit
cycle is hyperbolic and attractor.

By considering $V(x,y)=2x^2+y^4$ and $k=2/3$ equation \eqref{M}
becomes
$$
M_{\{V,\frac{2}{3}\}}=\frac{2}{3}\left((3-10m)x^2+my^4\right)y^4,
$$
which clearly  is non-negative on $\R^2$ for $m\in(0,3/10]$.

If we use the same $V(x,y)$ that in previous case, but
$k=K(m)=8(11m+R)/(10m+3)^2$, with $R=\sqrt{m(1-4m)(25m-9)}$ then we
have {\small $$ M_{\{V,K(m)\}}=\left(\frac{2
}{3+10m}\left(\frac{(m+R)(11m+R)}{m}\right)^{1/2}x^2
+\frac{2(3-10m)}{3+10m}\left(\frac{m(11m+R)}{(m+R)}\right)^{1/2}
y^4\right)^2,
$$}
hence $M_{\{V,K(m)\}}$ is non-negative on $\R^2$ for
$m\in(1/4,9/25]$.

Therefore system \eqref{sism} has no limit cycles for $m\in(0,9/25]$ as we wanted
to show.

To prove the non-existence of polycycles for $m\in(0,9/25)$ we use a
different approach. Following \cite{So}, we can associate to each
polycycle $\Gamma$, with $k$ hyperbolic saddles at its corners, the
 number $
 \rho(\Gamma)=\prod_{i=1}^k {b_i}/{a_i},
$ where $-a_i<0<b_i,$ $i=1,\ldots k,$ are the eigenvalues at the
saddles. Then, $\Gamma$ is stable (respectively, unstable) if
$\rho(\Gamma)<1$ (respectively, $\rho(\Gamma)>1$). In our case
\[\rho(\Gamma)=\frac { \left( 5\sqrt {m}-3+\sqrt {25m+18 \sqrt
{m}+9} \right) ^{4}}{48^2{m}}.\] Then, easy computations show that
the polycycle is an attractor if $m<9/25$ and a repeller if
$m>9/25$. Assume, to arrive to a contradiction, that for $m<9/25$
the polycycle exists. Then  both, the polycycle and the origin,
would be attractors. Applying the Poincar\'{e}-Bendixson Theorem we
could ensure that the system would have at least one periodic orbit
between them. This result  is in contradiction with the first part
of the proof, where the non-existence of periodic orbits is
established.

It only remains to show that for $m=9/25$ the polycycle neither
exists. To prove this fact we could study the stability of the
polycycle showing that if it exists it would be  attractor, arriving
again to a contradiction. Nevertheless it is easier to apply
Proposition~\ref{GenBenDul} with the $V$ and $k=K(9/25)$ used to
prove the non-existence of periodic orbits. Indeed, this later
approach, taking the corresponding $V$ and $k$, could also be used
for all values of $m\in(0,9/25]$, but we have preferred to include a
proof based on the study of the stability of the limit cycle and the
polycycle.
\end{proof}

\begin{lemma}\label{LeVM}
 Let $X$ be the vector field associated to system
\eqref{sism}.

\begin{itemize}

\item[(i)] If  we take $k=1/3$ and $V_1(x,y)=g_0(y)+g_1(y)x$ where
$g_0(y)=g_1'(y)$ and $g_1(y)$ a solution of the second order linear
ordinary differential equation
\begin{equation}\label{aku}
-g_1''(y)+my^5g_1'(y)-\frac{5}{3}my^4g_1(y)=0,
\end{equation}
then \eqref{M} reduces to the function
\begin{equation}\label{Maku}
M_1:=M_{\left\{V_1,\frac{1}{3}\right\}}(x,y)
=\frac{1}{3}\,{y}^{3}\left(3m{y}^{2}g_1''(y)-5myg_1'(y)+3\,g_1(y)\right).
\end{equation}

\item[(ii)] If we take $k=2/3$ and $V_2(x,y)=g_0(y)+g_1(y)x+g_2(y)x^2$,
with
\begin{align}\label{g0g1}
g_1(y)&=g_2'(y),\nonumber\\
g_0(y)&=(1/2)g_2''(y)-(1/2)my^5g_2'(y)+(5/3)my^4g_2(y),
\end{align}
then \eqref{M} becomes
\begin{align}\label{Mode3}
M_2&:=M_{\left\{V_2,\frac{2}{3}\right\}}(x,y)=\Big(-\frac{1}{2}g'''_2(y)
+\frac{3}{2}m{y}^{5}g''_2(y)-\frac{5}{2}m{y}^{4}g'_2(y)
+\frac{2}{3}(3-10m){y}^{3}g_2(y)\Big)x\nonumber
\\&+\frac{1}{18}{y}^{3}\left(9m{y}^{2}g'''_2(y)
-m(30+9my^6)yg''_2(y)+3(6+5{m}^{2}{y}^{6})g'_2(y)
+20{m}^{2}{y}^{5}g_2(y)\right).
\end{align}

\end{itemize}

\end{lemma}

\begin{proof}
(i) If  $V_1(x,y)=g_0(y)+g_1(y)\,x$ and $k=1/3$, then
\begin{align*}
M_1=&\langle\nabla V_1,X\rangle-\frac{1}{3}\Div(X)V_1\\
=&\left(g_0(y)-g_1'(y)\right)x^2+\Big(-g_0'(y)+my^5g_1'(y)-\frac{5}{3}my^4g_1(y)\Big)x\\
 &+\frac{1}{3}y^3\left(3mg_0'(y)y^2-5mg_0(y)y+3g_1(y)\right).
\end{align*}
By choosing $g_0(y)=g_1'(y)$ the coefficient of $x^2$ in $M_{1}$
vanishes, and we obtain {\small
\begin{align*}
M_{1} =&\left(-g_1''(y)+my^5g_1'(y)-\frac{5}{3}my^4g_1(y)\right)x
+\frac{1}{3}y^3\left(3mg_1''(y)y^2-5myg_1'(y)+3g_1(y)\right).
\end{align*}}
Finally, if $g_1(y)$ is a  solution of \eqref{aku} we get
\eqref{Maku}.

\smallskip\smallskip
(ii) If $k=2/3$ and $V_2(x,y)=g_0(y)+g_1(y)x+g_2(y)x^2$, then
\begin{align*}
M_{2}=&\langle\nabla V_2,X\rangle-\frac{2}{3}\Div(X)V_2\\
=&(g_1(y)-g_2'(y))x^3+\Big(my^5g_2'(y)-\frac{10}{3}my^4g_2(y)
-g_1'(y)+2g_0(y)\Big)x^2+\Big(2y^3g_2(y)\\
&+my^5g_1'(y)-\frac{10}{3}my^4g_1(y)-g_0'(y)\Big)x
+\frac{1}{3}y^3\big(3g_1(y)+3my^2g_0'(y)-10myg_0(y)\big).
\end{align*}
By choosing $g_1(y)=g_2'(y)$ and
$g_0(y)=(1/2)g_2''(y)-(1/2)my^5g_2'(y)+(5/3)my^4g_2(y)$ the
coefficients of $x^2$ and $x^3$ in $M_{2}$ vanish. Then we
have~\eqref{Mode3}.
\end{proof}
\begin{remark}
Notice that if  $g_2(y)$ is a solution of the linear ordinary
differential equation
\begin{equation}\label{ode3}
-\frac{1}{2}g'''_2(y)+\frac{3}{2}m{y}^{5}g''_2(y)-\frac{5}{2}m{y}^{4}g'_2(y)
+\frac{2}{3}(3-10m){y}^{3}g_2(y)=0,
\end{equation}
then \eqref{M} reduces to a function depending only of the variable
$y$.
\end{remark}

\begin{proposition}\label{35} For
$m\in[3/5,\infty)$, system \eqref{sism} has neither periodic orbits
nor polycycles.
\end{proposition}
\begin{proof} We want to apply Proposition~\ref{GenBenDul}, taking $k=1/3$
and  $V_1(x,y)=g_0(y)+g_1(y)x$, with $g_0$ and $g_1$ as in (i) of
Lemma~\ref{LeVM}. Applying the transformation $z=my^6/6$,
equation~\eqref{aku} becomes
\begin{equation*}\label{ku}
zg_1''(z)+\left(\frac{5}{6}-z\right)g_1'(z)+\frac{5}{18}g_1(z)=0,
\end{equation*}
which is a Kummer equation, see \cite[pp. 504]{AS}. A particular
solution of this equation is
$$
g_1(z)=z^{1/6}\sum_{j=0}^{\infty}\frac{(-\frac{1}{9})_j}{(\frac{7}{6})_j}\frac{z^j}{j!},
$$
where $(a)_j:=a(a+1)(a+2)\cdots(a+j-1)$ and $(a)_0=1$.
Therefore we consider
\begin{equation*}
g_1(y)=\left(\frac{m}{6}\right)^{1/6}y\sum_{j=0}^{\infty}\frac{(-\frac{1}{9})_j}{(\frac{7}{6})_j}
\left(\frac{m}{6}\right)^j\frac{y^{6j}}{j!},
\end{equation*}
which  is convergent on the whole $\R$ and satisfies \eqref{aku}.
Its  derivatives are
\begin{align*}
g_1'(y)&=\left(\frac{m}{6}\right)^{1/6}\sum_{j=0}^{\infty}\frac{(-\frac{1}{9})_j}{(\frac{7}{6})_j}
\left(\frac{m}{6}\right)^j(6j+1)\frac{y^{6j}}{j!},
\end{align*}
\begin{align*}
g_1''(y)&=\left(\frac{m}{6}\right)^{1/6}\sum_{j=0}^{\infty}\frac{(-\frac{1}{9})_j}{(\frac{7}{6})_j}
\left(\frac{m}{6}\right)^j6j(6j+1)\frac{y^{6j-1}}{j!}.
\end{align*}
Replacing  the above functions in \eqref{Maku} we obtain
\begin{align*}
M_{1}=&\left(\frac{3-5m}{3}\right)\left(\frac{m}{6}\right)^{1/6}y^4
\\&+\left(\frac{1}{3}\right)\left(\frac{m}{6}\right)^{1/6}\sum_{j=1}^{\infty}\frac{(-\frac{1}{9})_j}{(\frac{7}{6})_j}
\left(\frac{m}{6}\right)^j\left(\frac{1}{j!}\right)\big(m(6j+1)(18j-5)+3\big)y^{6j+4}.
\end{align*}
Since $(-\frac{1}{9})_j$ is negative for all $j$, it follows that
$M_{1}\le0$ for $m\geq 3/5$, and vanishes only on $y=0.$ Therefore
the result follows by applying Proposition~\ref{GenBenDul}.
\end{proof}

\section{Uniqueness and hyperbolicity of the  limit cycle for
$m\in(1/2,3/5)$}\label{se:uni}

In this section we prove that for $m\in(1/2,3/5)$,  system~\eqref{sism} has at most
 one limit cycle or one polycycle and both never coexist. Moreover, we show
 that when the limit cycle  exists, it is   hyperbolic. The
uniqueness of the limit cycle was already proved in \cite{Ga-Go}.  Our approach is
different and, like in the previous section, it is based on the construction of a
suitable Dulac function. This section ends with the proof of
Proposition~\ref{mteo2}.

\begin{lemma}\label{remcuad}
Let $\mathcal{S}$ be the open set bounded by the lines $x=\pm
m^{-1/4}$ and $y=\pm m^{-1/4}$ and let $\Omega$ be the connected
component containing the origin and bounded by the above four
straight lines  and the hyperbola $xy+1=0$, see Figure
\ref{cuadrado}. The following holds:
\begin{itemize}
\item[(i)]
The vector field $X$ associated to system \eqref{sism} is transversal to the
boundary $\partial \mathcal{S}$ of the square $\mathcal{S}$ except at the two
saddle critical points of system \eqref{sism}.
\item[(ii)] If system \eqref{sism} has a periodic orbit or a polycycle,
it must be contained in
$\Omega\subset\mathcal{S}$.
\end{itemize}
\end{lemma}

\begin{figure}[h]
\begin{center}
\begin{tabular}{cc}
\epsfig{file=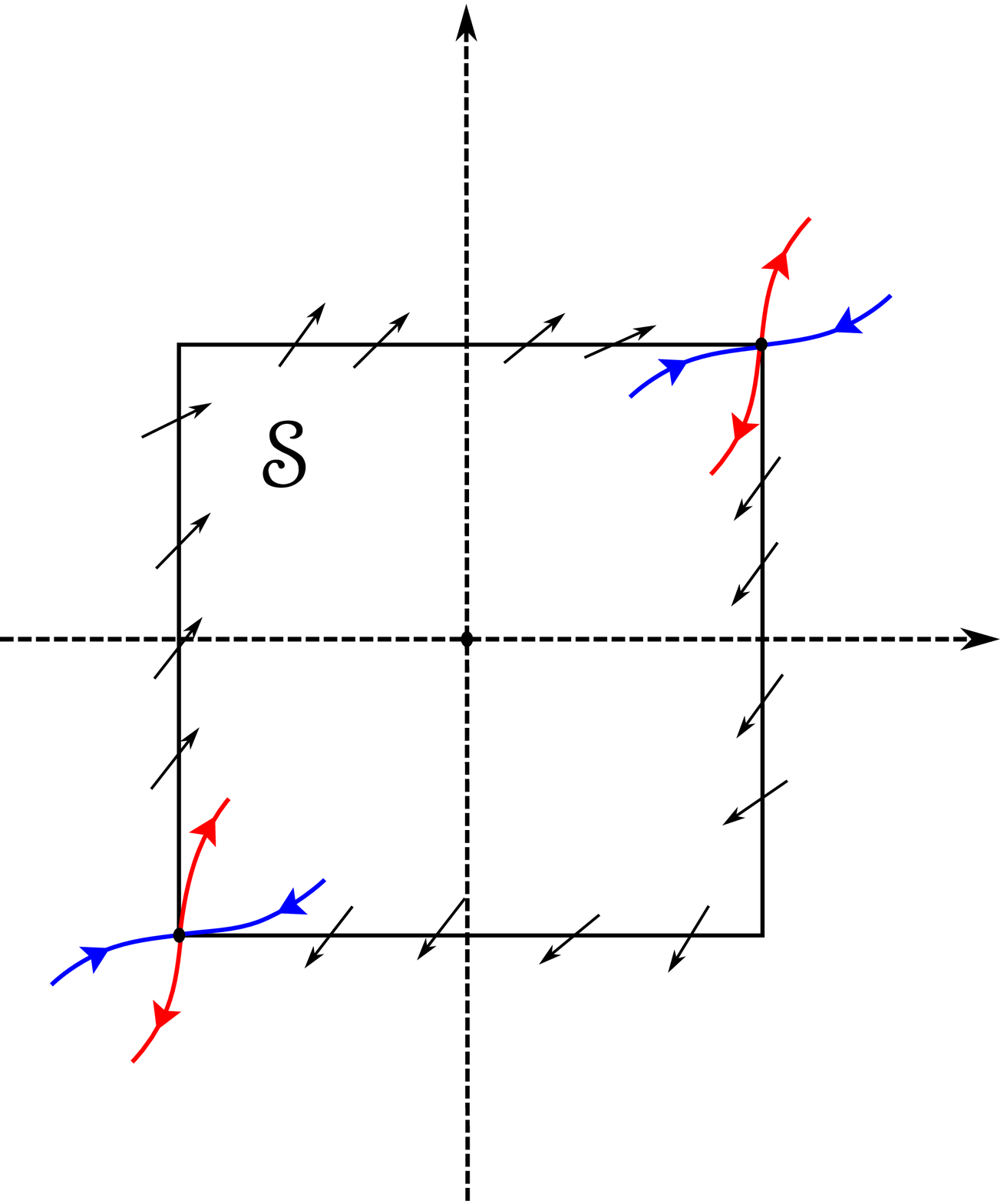,width=3.5cm}\qquad\qquad&\qquad\qquad
\epsfig{file=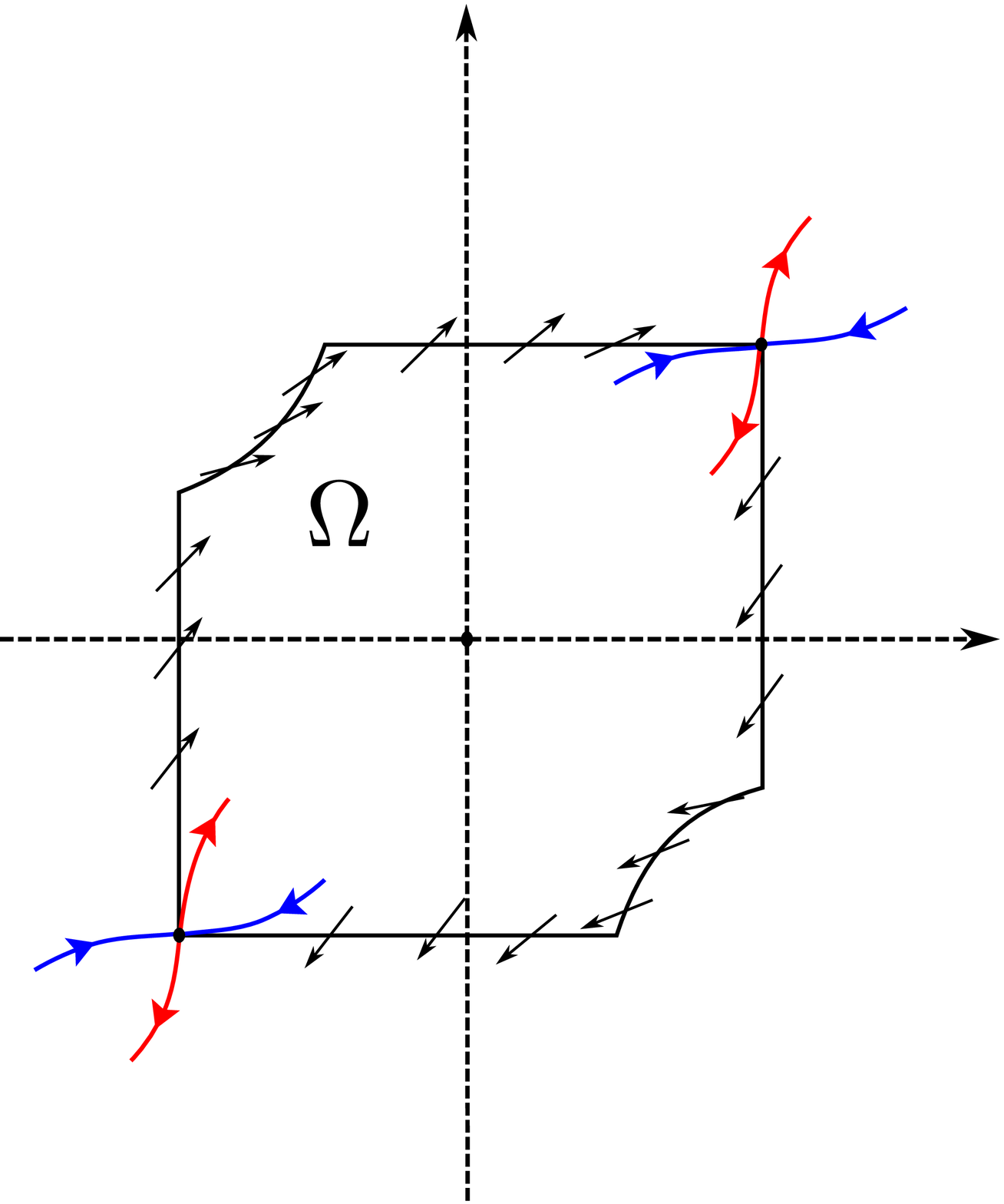,width=3.5cm}\\
(a)\qquad\qquad & \qquad\qquad (b)
\end{tabular}
\caption{Regions $\Omega$ and $\mathcal S$.} \label{cuadrado}
\end{center}
\end{figure}

\begin{proof}(i).
Consider the function $f(x,y)=x-m^{-1/4}$. It is not difficult to
see that $\langle \nabla f, X\rangle$ restricted to $x-m^{-1/4}=0$
has the expression $y^3-m^{-1/4}$ which is negative for
$y\in(-m^{-1/4},m^{-1/4})$. Analogously, we can see that the
direction of $X$ along $\partial \mathcal{S}$ is as showed in Figure
\ref{cuadrado} (a).

(ii). It is well-known that the sum of the indices of all the
singularities surrounded  by  a periodic orbit, or a polycycle is
one. Recall that the indices of the saddle points are $-1$ and the
index of a monodromic point is $+1$. Hence, if a periodic orbit or a
polycycle  $\Gamma$  exist they must surround only the origin.
Moreover, by statement~(i), $\Gamma$ cannot intersect $\partial
\mathcal{S}$. Finally, a simple computation shows that
$\langle\nabla(xy+1), X\rangle$ restricted to $xy+1=0$ is
$(1-m)/x^4$, which implies that $X$ is transversal to $xy+1=0$.
Hence $X$ is transversal to $\partial \Omega$ and the lemma follows.
\end{proof}

\begin{proposition}\label{uniq}
For $m\in[1/2,3/5)$, system~\eqref{sism}  has at most
 one limit cycle and one polycycle and both never coexist.
Moreover, when the limit cycle exists  it is hyperbolic and repeller.
\end{proposition}
\begin{proof}
Following statement (ii) of Lemma \ref{LeVM} we take $k=2/3$ and a
function $V_2(x,y)=g_0(y)+g_1(y)x+g_2(y)x^2$ adequate to apply
Proposition~\ref{GenBenDul} for proving the uniqueness of the limit
cycles or polycycles for system~\eqref{sism}.

We will take $g_{2}(y)$ as a truncated Taylor series at the origin of a suitable
solution of \eqref{ode3} such that the curve $\{V_2=0\}$ has an oval surrounding
the origin, and that $M_2$ does not change sign in $\Omega$. These two properties
will imply the result.

The general solution of \eqref{ode3} is the linear combination of generalized
hypergeometric functions
\begin{align}\label{general} g_2(y)=&\,
C_0\sum_{j=0}^{\infty}\frac{\left(\phi^+(m)\right)_j\left(\phi^-(m)\right)_j}
{\left(\frac{2}{3}\right)_j\left(\frac{5}{6}\right)_j}\left(\frac{m}{2}\right)^j
\frac{y^{6j}}{j!}
+C_1y\sum_{j=0}^{\infty}\frac{\left(\varphi^+(m)\right)_j\left(\varphi^-(m)\right)_j}
{\left(\frac{5}{6}\right)_j\left(\frac{7}{6}\right)_j}\left(\frac{m}{2}\right)^j
\frac{y^{6j}}{j!}\nonumber
\\&+C_2y^2\sum_{j=0}^{\infty}\frac{\left(\psi^+(m)\right)_j
\left(\psi^-(m)\right)_j}
{\left(\frac{7}{6}\right)_j\left(\frac{4}{3}\right)_j}\left(\frac{m}{2}\right)^j
\frac{y^{6j}}{j!},
\end{align}
where $ \phi^\pm(m)= \pm A(m)-2/9,$ $\varphi^\pm(m)= \pm A(m)-1/18,$
$\psi^\pm(m)= \pm A(m)+1/9, $ with $A(m)=\sqrt{(14m-3)/m}/9$.

We look for an even solution, so we take $C_1=0$. As we will
consider $C_0\ne0,$ it is not restrictive to choose $C_0=1$.
Finally, the constant $C_2=-(3/5-m)^{2/3}$ is fixed after some
previous simulations and taking into account that we already know
that at $m=3/5$ there is a Hopf-like bifurcation.

Once we have fixed the above constants, we calculate the Taylor polynomial of
degree~12 of $g_2$ at $y=0,$ $\mathcal{T}_{12}(g_2)$, obtaining
\begin{align}\label{g2final}
\mathcal{T}_{12}(g_2(y))=&\frac{1}{89100}(3-10m)(3+35m)y^{12}
-\frac{1}{6300}(75-125m)^{2/3}(3-13m)y^8\nonumber
\\&+\frac{1}{90}(3-10m)y^6-\frac{1}{25}(75-125m)^{2/3}y^2+1.
\end{align}
So,  in (ii) Lemma~\ref{LeVM}, we fix $g_2$ as
$\mathcal{T}_{12}(g_2(y))$. Then the corresponding $g_0$ and $g_1$
are given by~\eqref{g0g1}. Thus, $M_2$ is of the form
$M_2=(\phi(y)x+\psi(y))y^4$ where

\begin{align*}
\phi(y)=&{\frac {1}{9450}}\Big({\frac {7}{99}}\, \left(3-10\,m
\right) \left( 242\,m+ 3 \right) \left( 35\,m+3 \right) {y}^{11}\\
&\qquad\quad+\left( 75-125\,m \right) ^{2/3} \left( 86\,m+3 \right)
\left(
13\,m-3 \right) {y}^{7} \Big),\\
\psi(y)=&-{\frac {247}{400950}}\,{m}^{2} \left( 3- 10\,m \right)
\left( 35\,m+3 \right) {y}^{16}-{\frac {13}{4050}}\,{m} ^{2} \left(
75-125\,m \right) ^{2/3} \left( 13\,m-3 \right) {y}^{12}\\&+{ \frac
{1}{7425}}\, \left(3-10\,m \right) \left( 550\,{m}^{2}+145\,m +3
\right) {y}^{10}\\&+{\frac {2}{4725}}\, \left( 75-125\,m \right)
^{2/3 } \left(196\,{m}^{2}-45\,m -9\right){y}^{6}\\&+\frac{1}{15}
\left(3-5\,m \right) {y}^{4}-{\frac {2}{75}}\, \left( 75-125\,m
\right) ^{2/3} \left( 3-5\,m \right).
\end{align*}

The proposition follows if we prove that $M_2$ does not change sign
on the region~$\Omega$. In fact, it is sufficient to prove that
$M:=M_2/y^4$ does not change  sign on $\Omega$.

The idea is to show that $\{M=0\}$ does not intersect $\Omega$.
Since $M$ is linear in the variable $x$, $\{M=0\}$ cannot have ovals
inside $\Omega$. If $\{M=0\}$ has a component in $\Omega$, this
component would have to cross $\partial \Omega$ by continuity of the
function. Then, it suffices to see that $\{M=0\}$ does not intersect
$\partial \Omega$. Moreover, as $M$ satisfies $M(x,y)=M(-x,-y)$, it
is sufficient to study $M$ on half of $\partial \Omega$. To deal
only with polynomials we introduce the new variables $n=\sqrt[4]m$
and $s=(75-125m)^{2/3}$. Notice that $s^3=(75-125n^4)^2.$

We split the half of the boundary of $\Omega$  in four pieces:
\begin{itemize}
\item The segment $\gamma_1=\{(x,1/n)\,:\, -n< x< 1/n\},$

\item The segment $\gamma_2=\{(1/n,y)\,:\, -n< y< 1/n\},$

\item The piece of hyperbola $\gamma_3=\{(x,-1/x )\,:\, n< x< 1/n\},$

\item The corners $\gamma_4=\{(1/n,1/n ),(1/n,-n),(n,-1/n)\}$

\end{itemize}
and we have to prove that $\{M=0\}\cap \gamma_i=\emptyset$ for each
$i=1,2,3,4.$

These facts can be seen proving that for $n\in
I:=[\sqrt[4]{1/2},\sqrt[4]{3/5}\,)$,
\begin{itemize}
\item $Q_1(x,n,s):=M(x,1/n)\ne0,$ for $x\in(-n,1/n)$.

\item $Q_2(y,n,s):=M(1/n,y)\ne0,$ for $y\in(-n,1/n)$.

\item $Q_3(x,n,s):=M(x,-1/x)\ne0,$ for $x\in(n,1/n)$.

\item $M(1/n,1/n )\cdot M(1/n,-n)\cdot M(n,-1/n)\ne0.$
\end{itemize}

Lemma~\ref{ll}, with $r=0$, is a convenient tool to prove the first
three items. The proof of the last item is an straightforward
consequence of Sturm method.

 We will give  the details of the proof
that $ Q_2(y,n,s)\ne0$, which is the most elaborate case. The
remainder two cases follow similarly.

Writing $Q(y,n,s):=2806650 n Q_2(y,n,s) $ we get that
\begin{align*}
Q(y,n,s)=&
1729\,n^9(35n^4+3)(10n^4-3)y^{16}-9009\,n^9s(13n^4-3)y^{12}\\&-21(10n^4-3)
(242n^4+3)(35n^4+3)y^{11}\\
&-378n(10n^4-3)(550n^8+145n^4+3)y^{10}+297s(13n^4-3)(86n^4+3)y^7\\&+1188ns(196n^8-45n^4-9)y^6
-187110n(5n^4-3)y^4 +74844ns(5n^4-3).
\end{align*}
Looking at Lemma~\ref{ll} with $r=0$, it suffices to prove  the following three
facts:
\begin{enumerate}[(i)]
\item When $n=\sqrt[4]{1/2}\in I$,  $Q(y,n,s)\ne0$ for $y\in(-n,1/n)$.

\item For $n\in I$, $\triangle_y Q(y,n,s)\ne0$.

\item   For $n\in I$, $Q(-n,n,s)\cdot Q(1/n,n,s)\ne0$.

\end{enumerate}

Since the polynomial has no rational coefficients the proof of item
(i) requires some special tricks.  Notice that when
$n=\sqrt[4]{1/2}$ then $s=5\sqrt[3]{10}/2$. Hence,

\begin{align*}R(y):= Q\Big(y,\frac1{\sqrt[4]{2}},\frac52\sqrt[3]{10}\Big)=&{\frac {70889}{8}}
\sqrt[4]{8}\,{y}^{16}-{\frac {315315}{32}} \sqrt[4]{8}
\sqrt[3]{10}\,{y}^{12}-106764 {y}^{11}\\&-80514
\,\sqrt[4]{8}\,{y}^{10}+{ \frac {239085}{2}}
\sqrt[3]{10}\,{y}^{7}+{\frac {51975}{2}} \sqrt[4]{8}
\sqrt[3]{10}\,{y}^{6}\\&+{\frac {93555}{2}}
\sqrt[4]{8}\,{y}^{4}-{\frac { 93555}{2}} \sqrt[4]{8}\sqrt[3]{10}.
\end{align*}

%

We will prove that the above polynomial has no real roots in
$[-1,12/10]\supset(-n,1/n).$ The Sturm method gives polynomials with
huge coefficients and our computers have problems to deal with them.
We use a different approach. We know, that
\[
\underline{n}:=\dfrac{3002}{1785} <\sqrt[4]{8}<
\dfrac{37}{22}=:\overline{n}, \quad \underline{s}:=\dfrac{28}{13}
<\sqrt[3]{10}< \dfrac{265}{123}=:\overline{s},
\]
where these four rational approximations are obtained computing the continuous
fraction expansion of both irrational numbers. If we construct the polynomial, with
rational coefficients,
\begin{align*}R^+(y)=&{\frac {70889}{8}}
\overline{n}\,{y}^{16}-{\frac {315315}{32}} \underline{n}\,
\underline{s}\,{y}^{12}-106764 {y}^{11}-80514 \underline{n}\,{y}^{10}\\&+{ \frac
{239085}{2}}  \overline{s}\,{y}^{7}+{\frac {51975}{2}} \overline{n}\,
\overline{s}\,{y}^{6}+{\frac {93555}{2}} \overline{n}\,{y}^{4}-{\frac {
93555}{2}}\underline{n} \,\underline{s},
\end{align*}
it is clear that for $y\ge0,$  $R(y)<R^+(y).$ In fact,
\begin{align*} R^+(y)=&{\frac
{2622893}{176}} {y}^{16}-{\frac {2427117}{68}} {y}^{12}- 106764 {y}^{11}-{\frac
{11509668}{85}} {y}^{10}\\&+{\frac {21119175}{82 }} {y}^{7}+{\frac {15442875}{164}}
{y}^{6}+{\frac {314685}{4}} {y}^ {4}-{\frac {37446948}{221}}
\end{align*}
and, now, using the Sturm method it is quite easy to prove that
$R^+(y)<0$ for $y\in[0,12/10].$ Hence, in this interval,
$R(y)<R^+(y)<0,$ as we wanted to prove.

To study the values of $y<0$ we construct a similar upper bound,
\begin{align*}R^-(y)=&{\frac {70889}{8}}
\overline{n}\,{y}^{16}-{\frac {315315}{32}} \underline{n}\,
\underline{s}\,{y}^{12}-106764 {y}^{11}-80514 \underline{n}\,{y}^{10}\\&+{ \frac
{239085}{2}}  \underline{s}\,{y}^{7}+{\frac {51975}{2}} \overline{n}\,
\overline{s}\,{y}^{6}+{\frac {93555}{2}} \overline{n}\,{y}^{4}-{\frac {
93555}{2}}\underline{n} \,\underline{s},
\end{align*}
and applying the same method the result follows.

To prove (ii) we compute
$$
\triangle_y Q(y,n,s)=n^{42}s^3(5n^4-3)^5(35n^4+3)^3(10n^4-3)^3
P_{258}(n,s),
$$
where $P_{258}(n,s)$ is a polynomial in $n$ and $s$ of degree 258. Clearly, the
roots of the first five factors of the above discriminant  are no relevant for our
problem because the corresponding $n$ is not in $I.$ To study whether
$P_{258}(n,s)$ vanishes or not we compute
\[
\Res (P_{258}(n,s),(75-125n^4)^2-s^3,s)=(5n^4-3)^{24}{P}_{390}(n^2),
\]
where ${P}_{390}(n^2)$ is a polynomial of degree 390 in $n^2$.
Applying again the Sturm method  we get that ${P}_{390}(n^2)$ has no
significant roots for our study. Finally, the numerator of
$Q(-n,n,s)\cdot Q(1/n,n,s)$ is a polynomial in $n$ and $s$ of degree
49. Using the same trick as above we prove item (iii). In this case
the polynomial we have to deal with has degree  152 in $n$.

Therefore $\{M=0\}\cap \partial\Omega=\emptyset$ and as a consequence $\{M=0\}\cap
\Omega=\emptyset$.

Finally,  it is not difficult to see, because $V$ is quadratic in
$x,$ that the set $\{V(x,y)=0\}$ has exactly one oval surrounding
the origin. Hence, the proposition follows.
\end{proof}

\begin{proof}[Proof of Proposition~\ref{mteo2}] Notice that the  function $V$ used
in the proof of Proposition~\ref{uniq} coincides with the function $V(x,y,m)$ of
the statement of the proposition. Taking $k=2/3$ we are also under the hypotheses
of Corollary~\ref{co:basin}. Therefore the set $\mathcal{U}_m$ is contained in
$\mathcal{W}^s_{\bf 0}$, as we wanted to prove.
\end{proof}

We remak that following similar ideas that in the above proof  we
can construct bigger sets contained in $\mathcal{W}^s_{\bf 0}$. For
a  given $m$, let us denote by $ \mathcal{T}_{\ell}(g_2(x;C_2))$ the
Taylor polynomial of degree $\ell$ at $x=0$, of the
function~\eqref{general} with $C_0=1,$ $C_1=0.$ Then for each
$\ell\in\mathbb{N}$ and $C_2\in\mathbb{R}$ we can take this function
as a new seed $g_2$ for constructing the corresponding $V$ as in
(ii) of Lemma~\ref{LeVM}. Then checking that the oval contained in
$\{V=0\}$ is crossed inwards by the flow of the system, the result
follows for the function~$V$ constructed with these $\ell$
and~$C_2$.

\section{Non-existence of limit cycles and polycycles for
$m\in(9/25,0.547]$}\label{se:ne2}

This section contains  new non-existence results for system \eqref{sism}. We split
the interval into the subintervals $(9/25,1/2)$ and $[1/2,0.547]$. Recall that our
numerical study shows that the system has no limit cycles for $m<0.56011\ldots$ As
$m$ becomes closer to this bifurcation value the proof of non-existence of periodic
orbits and polycycles becomes harder.

\begin{proposition}\label{925}
For $m\in(9/25,1/2)$, system~\eqref{sism}  has neither limit cycles nor polycycles.
\end{proposition}

\begin{proof} We would like to apply Proposition~\ref{GenBenDul}. To this end we
will follow similar steps that in the proof of
Proposition~\ref{uniq}, but with a function $V$ such that the set
$\{V=0\}$ has no oval in $\Omega.$ Recall that $\Omega$ is  the
domain introduced in Lemma~\ref{remcuad}, where the limit cycles and
the polycycles must lay. We take
$V=V_2(x,y)=g_0(y)+g_1(y)x+g_2(y)x^2$ with $g_1(y)=g_2'(y)$,
$g_0=(1/2)g_2''(y)-(1/2)my^5g_2'(y)+(5/3)my^4g_2(y)$. Now we
consider $ g_2(y)=a_0+a_{2} y^{2}+a_{4} y^{4}+a_{6} y^{6}+a_{8}
y^{8}$, with coefficients to be determined. From statement~(ii) of
Lemma~\ref{LeVM} it follows that the corresponding $M_2$ is a
polynomial function in $x$ of the form $M_2=\phi(y)x+\psi(y)$ where
$\phi(y)$ and $\psi(y)$ are polynomials in the variable $y$ whose
coefficients depend on $a_{2j}$, $j=0,1,\ldots,4$. In order to
simplify the computations, we change the parameter $m$ by $n^4$ to
transform $V$ into a polynomial in the variables $x$, $y$, and $n$.
Since $m\in(9/25,1/2)$ we can restrict our study to
$n\in(0.77,0.844)$.

We consider the values of $a_4,a_6$ and $a_8$  such  that $\phi(y)$
has a zero at $y=0$ of multiplicity nine,  we choose the value of
$a_2$  by imposing that $M_2$ vanishes at the two saddle points of
the system and, finally, we use the freedom of changing $g_2(y)$ by
$\lambda g_2(y),$ for any $0\ne\lambda\in\R$,  to remove all the
denominators. We obtain that
\begin{align*}
{g}_2(y)=&270 (9 + 51 n^2 + 213 n^4 + 535 n^6 + 924 n^8 + 756
n^{10})\\& - 756 n^2 (9 + 42 n^2 + 105 n^4 + 130 n^6) y^2
\\
&+3 (3 - 10 n^4) (9 + 51 n^2 + 213 n^4 + 535 n^6 + 924 n^8 +
       756 n^{10}) y^6\\&
    -3 n^2 (3 - 13 n^4) (9 + 42 n^2 + 105 n^4 + 130 n^6) y^8.
\end{align*}
The corresponding $M_2$ is of the form
\begin{align}\label{eq:m}
M_2(x,y)=\frac{2}{3}y^4\,(\phi(y)x+\psi(y))=:\frac{2}{3}y^4\,M(x,y),
\end{align}
where
\begin{align*}
\phi(y)=&3 (3 - 10 n^4) (3 + 35 n^4) (9 + 51 n^2 + 213 n^4 + 535 n^6 + 924 n^8 +
756 n^{10}) y^5
\\&- 3 n^2 (3 - 13 n^4) (3 + 86 n^4) (9 + 42 n^2 + 105 n^4 + 130 n^6) y^7,\\
\psi(y)=&-756 n^2 (3 - 5 n^4) (9 + 42 n^2 + 105 n^4 + 130 n^6) \\&+
 27 (3 - 5 n^4) (9 + 51 n^2 + 213 n^4 + 535 n^6+ 924 n^8 +
    756 n^{10}) y^4 \\&-
 12 n^2 (9 + 42 n^2 + 105 n^4 + 130 n^6) (9 + 45 n^4 -
    196 n^8) y^6 \\&-
 40 n^8 (3 - 10 n^4) (9 + 51 n^2 + 213 n^4 + 535 n^6
 + 924 n^8 + 756 n^{10}) y^{10}\\&+
 91 n^{10} (3 - 13 n^4) (9 + 42 n^2 + 105 n^4 + 130 n^6) y^{12}.
\end{align*}

Recall that the main hypothesis in Proposition~\ref{GenBenDul} is that $M$ does not
change on $\Omega$. As we will see, this happens only for $n\in J:=(0.77,\widetilde
n ]$ where $\widetilde n\approx 0.8045592  $ will be precisely defined afterwards.
When $n\in K:=(\widetilde n,0.844)$ the result will be a consequence of the
variation of Proposition~\ref{GenBenDul} described in Remark~\ref{GenBenDul2}.

For $n\in J$, following similar  steps that in the proof of Proposition~\ref{uniq},
we divide  half of the boundary of $\Omega$  in five pieces:
\begin{itemize}
\item The segment $\gamma_1=\{(x,1/n)\,:\, -n< x< 1/n\},$

\item The segment $\gamma_2=\{(1/n,y)\,:\, -n< y< 1/n\},$

\item The piece of hyperbola $\gamma_3=\{(x,-1/x )\,:\, n< x< 1/n\},$

\item The corners $\gamma_4=\{(1/n,-n),(n,-1/n)\}$,

\item The corner $\gamma_5=\{(1/n,1/n)\}$

\end{itemize}
and we will prove that $\{M=0\}\cap \gamma_i=\emptyset$ for each $i=1,2,3,4$ and
that although $(1/n,1/n)\in\partial\Omega$, the set $\{M=0\}$ does not enter in
$\Omega.$ From these results we will have proved that $M$ does not change sign on
$\Omega$ and, as a consequence, the proposition will follow for $n\in J.$

To prove the fifth assertion it suffices to study the function $M$
in a neighborhood of the point $(1/n,1/n)\in\partial \Omega$. By the
construction of $M$, it holds that  $M(1/n,1/n)=0$. By computing the
partial derivatives  of $M$ at this point we obtain which is the
tangent vector of the curve at  $(1/n,1/n)$. Then, it is easy to see
that when $n\in J$,  in a punctured neighborhood $\mathcal W$ of
$(1/n,1/n)$, it holds that $\mathcal{W}\cap\{M=0\}\cap
\,\Omega=\emptyset$. In fact, $\widetilde n\in\partial J$ is a
solution of the equation
$$\numer\Big(\frac{\partial M(x,y)}{\partial x}\Big|_{(x,y)=(1/n,1/n)}\Big)=0,$$ where
$\numer(\cdot)$ denotes the numerator of the rational function.
Moreover,
\begin{align}\label{eee}
M\Big(x,\frac1 n\Big )=-\frac{9( nx-1 )}{n^4} & \Big( 88200 {n}^{16}+107800
{n}^{14}-4930 {n}^{12}\nonumber\\&- 37380 {n}^{10}-15855 {n}^{8}-2736 {n}^{6}+576
{n}^{4}+108 {n}^{2} -27 \Big)
\end{align}
and  $\widetilde n$ is also the positive root of the polynomial in
$n$ appearing in the right-hand side of the above formula. Notice
that when $n=\widetilde n$, the straight line $\{y=1/\widetilde n\}$
is a subset of $\{M=0\}.$  This fact is the reason for which this
approach only works for $n\in J=(0.77,\widetilde n]$.

Let us prove the remainder four assertions. As in the proof of
Proposition~\ref{uniq}, they follow by showing that when $n\in J$,
\begin{itemize}
\item $R_1(x,n):=\numer(M(x,1/n))\ne0,$ for $x\in(-n,1/n)$.

\item $R_2(y,n):=\numer(M(1/n,y))\ne0,$ for $y\in(-n,1/n)$,

\item $R_3(x,n):=\numer(M(x,-1/x))\ne0,$ for $x\in(n,1/n)$.

\item $M(1/n,-n)\cdot M(n,-1/n)\ne0.$
\end{itemize}

That  $R_1$ has  no zeros in $J$, is an straightforward  consequence
of~\eqref{eee}.

To study $R_2$ and $R_3$ we will use  Lemma~\ref{remcuad}. We start computing the
discriminants,
\[
S_2(n)=\triangle_y(R_2(y,n)),\quad S_3(n)=\triangle_x(R_3(x,n)),
\]
and analyze whether they vanish or not on~$J.$ Using  the Sturm
method we get  that on~$J$,  $S_2$  vanishes only at one value
$n_2\approx0.8040188$ and  $S_3$ also vanishes only at one value
$n_3\approx 0.8045576.$   The root $n_2$ of $S_2$ forces to split
the study of $R_2(y,n)$ in the three subcases: $n\in(0.77,n_2)$,
$n=n_2$ and $n\in(n_2,\widetilde n]$. Doing the same type of
computations and reasoning as in the previous section we can prove
all the above assertion when $n\ne n_2$. The case $n=n_2$ follows by
continuity arguments, because it can be seen that in this situation
$R_2$ has a real multiple root but it is not in $(-n_2,1/n_2)$. The
study of $R_3$ is similar to the one of $R_2$ and we omit it. We
also get that $R_3$ neither vanishes on $(n,1/n).$

That for $n\ne \widetilde n$, $M(1/n,-n)\cdot M(n,-1/n)\ne0$  is
once more a consequence of the Sturm method.

 Therefore, when $n\in J,$ we are under the hypotheses of
 Proposition~\ref{GenBenDul},
 and we will know that the system has no limit cycles once
 we have proved that the set $\{V=0\}$ has no ovals. We defer the proof of this
  fact until we have considered the case $n\in K=(\widetilde n,0.844)$.

When $n\in K,$  we know that $\{M=0\}\cap \,\Omega\ne\emptyset$ and we are no more
under the hypotheses of Proposition \ref{GenBenDul}. Let us see that we can apply
the ideas of Remark \ref{GenBenDul2}. To this end we have to prove that
$\{M_2=0\}^*\cap \Omega $ is without contact for the flow of $X$. Note that
$\{M_2=0\}^*=\{M=0\}^*$.

We need to show that $\dot{M} =\langle\nabla M ,X\rangle$ does not vanish on $\{M
=0\}^*\cap \Omega$. We study the common points of $\{M =0\}$ and $\{\dot{M}=0 \}$
and prove that they are not in $\Omega.$ First, we compute
\[
\dot M(x,y)=\langle \nabla M(x,y),X(x,y)\rangle=:y^3N(x,y),
\]
and we remove the factor $y^3.$ We do not care about the points on $\{y=0\}$
because
\[M(x,0)=756 {n}^{2} \left( 5 {n}^{4}-3 \right)  \left(
{9}+42 {n }^{2}+105 {n}^{4}+130 {n}^{6} \right)\ne0,
 \]
for $n\in(0,0.88].$

The resultant $\Res(M ,N , x)$ factorizes as
\begin{equation*}\label{res}
\Res(M ,N , x)=y^{2}(n^2y^2-1) (P_{n,2}(y))(P_{n,34}(y)),
\end{equation*}
where $P_{n,2}(y)$ and $P_{n,34}(y)$ are polynomials in the variable $y$ with
respective degrees 2 and 34  and  whose coefficients are polynomial functions with
rational coefficients in the variable $n$.

Clearly, $(n^2y^2-1)$ does not vanish on $-1/n<y<1/n$. By using once
more Lemma~\ref{ll}  it is not difficult to prove  that $P_{n,2}(y)$
does not vanish either on $-1/n<y<1/n$, for $n\in(\widetilde
n,0.844)$. Hence we will focus on the factor $P_{n,34}(y)$.

We will use again Lemma~\ref{ll}.  By using  the Sturm method we get that
$\triangle_y(P_{n,34}(y))$ has no zeros in the interval $K$. In fact one zero is
$\widetilde n\in\partial K$ and another one is
 $n^*\approx
0.8445\not\in J$ and this is the reason  for which we can only prove the result
until $n=0.844<n^*.$ By using Sturm method, it can be shown that
$P_{n,34}(-1/n)\cdot P_{n,34}(1/n) \not=0$ for all $n\in K$ and, for instance, for
$n=n_0=83/100\in K$, the polynomial $P_{n_0,34}(y)$ has exactly two (simple) zeros
in $-1/n_0<y<1/n_0$. Then, Lemma~\ref{ll} with $r=2,$ implies that $P_{n,34}(y)$
has exactly two (simple) zeros in  $-1/n<y<1/n$, for all $n\in K$. We call them
$y=y_i(n),i=1,2$ and they are continuous function of $n$. Therefore, we need to
prove that the corresponding points in $\{M =0\}\cap\{N =0\}$ are outside of
$\Omega$.

Notice that because of the expression of $M$, given in~\eqref{eq:m},
the points in $\{M=0\}$ are on the curve
$\Gamma=\{\big(-\frac{\psi(y)}{\phi(y)},y\big)\,:\,
y\in\mathbb{R}\setminus\{0\}\}$. Moreover it can be easily seen that
$\phi(y)\ne0$ on the region that we are considering. Therefore the
points in  $\{M =0\}\cap\{N =0\}$ are given by the two continuous
curves
\[
\gamma_i:=\Big\{\Big(-\dfrac{\psi(y_i(n))}{\phi(y_i(n))},y_i(n)\Big)\,:\,n\in
K\Big\},\quad i=1,2.
\]
 For a fixed $n\in K$ it is not difficult to prove that the points in $\gamma_i,i=1,2$ are
outside of $\Omega$. If for some $n\in K$ there was a point inside $\Omega$,  by
continuity it would be at least one  point in one of the pieces of boundary of
$\Omega$ formed  by the straight  line  $\{x-1/n=0\}$ and the hyperbola
$\{xy+1=0\}$. To prove that such a point does not exist we compute the following
two resultants
\begin{align*}
&\Res\Big(\numer\Big( -\dfrac{\psi(y)}{\phi(y)}-\dfrac 1 n\Big), P_{n,34}(y),y
\Big)= P_{1250}(n),
\\
&\Res\Big(\numer\Big( -y\dfrac{\psi(y)}{\phi(y)}+1\Big), P_{n,34}(y),y \Big)=
P_{1260}(n),
\end{align*}
where $P_\ell(n)$ are given polynomials with rational coefficients
and degree $\ell$. Both polynomials factorize in several factors
and, using once more the Sturm method, we can easily prove that do
not vanish on $K$. Hence,  $\{M =0\}\cap\{ N
=0\}\cap\Omega=\emptyset$ which implies that $\{M =0\}\cap\Omega$ is
without contact by the flow of $X$, as we wanted to prove.

 Since $M $ is linear in the
variable $x$, $\{M =0\}$ cannot have ovals. Therefore, by
Remark~\ref{GenBenDul2}, to end the proof we need to show that the
set $\{V=0\}$ has no ovals either in $\Omega.$ We claim that the set
$\{V=0\}\cap\,\Omega$ is without contact by the flow of the system.
If this happens and $\{V=0\}$ had an oval then it would be without
contact. Then by the Poincar\'{e}-Bendixson Theorem it should surround
the origin. However, by considering the straight line passing
through the origin $y=9x/10$ it is easy to prove, by using again
Lemma~\ref{ll}, that the function $V (x,9x/10)$ does not vanish on
the interval $-1/n<x<1/n$ for all $n\in(0.77,0.844)$. Thus, $\{V
=0\}\cap \{y-9x/10=0\}=\emptyset$. Hence, $V $ has no ovals inside
$\Omega$ as we wanted to see and the proposition follows by using
all the above results and  the reasoning explained in
Remark~\ref{GenBenDul2}.

To prove the above claim, it suffices to see that  $\{M =0\}\cap\{
V=0\}\cap\Omega=\emptyset$. This is because precisely, $M\big|_{\{V=0\}}=\dot V.$

Recall that when $n\in J=(0.77,\widetilde n ]$ then $\{M =0\}\cap\Omega=\emptyset$
and so the result follows.

Let us consider the case $n\in K=(\widetilde n, 0.844)$. To study if
$\{V =0\}$ and $\{M =0\}$ intersect, we compute the resultant of $M
$ and $V $ with respect to $x$. We have
$$\Res(V ,M , x)=(n^2y^2-1)P_{n,30}(y),$$
where $P_{n,30}(y)$ is a polynomial of degree 30 and whose
coefficients are polynomial functions  in the variable $n$ with
rational coefficients. We want to prove that $\Res(V ,M , x)$ does
not vanish on the interval $-1/n<y<1/n$ for $n\in K$. It suffices to
study $P_{n,30}(y)$. We will use once more Lemma~\ref{ll}.

The polynomial  $P_{n,30}(-1/n)\cdot P_{n,30}(1/n)$   has  no real
roots when $n\in K$. Moreover hypothesis~(i) of Lemma~\ref{ll} holds
with $r=0$ (no real roots) by considering for instance $n_0=82/100$.
To see that condition (iii) of the lemma holds, we compute
$\triangle_y(P_{n,30}(y))$. It is a polynomial of degree $2728$ in
the variable $n$ which factorizes in several factors, being the
largest one of degree 594. From this decomposition we can  prove
that $\triangle_y(P_{n,30}(y))$ has no zeros  for $n\in K$.
Therefore, by Lemma~\ref{ll} we conclude that $P_{n,30}(y)$ does not
vanish on the whole interval $-1/n<y<1/n$ for $n\in K$, and the
claim follows.
\end{proof}

\begin{proposition}\label{547}
For $m\in[0.5,0.547]$, system~\eqref{sism}  has neither limit cycles nor
polycycles.
\end{proposition}

\begin{proof} We will construct a positive invariant region $\mathcal R$ having the two
saddle points in its boundary.  As we will see, the proposition
follows once we have constructed this region, simply by using the
uniqueness and hyperbolicity of the limit cycle, whenever it exits.
We remark that in this proof we will not use the Bendixson-Dulac
theorem.

 Assume that such a positive invariant region $\mathcal R$ exits.
 By the Index theory, if the system had a limit cycle, it
 should surround only the origin. By Proposition~\ref{uniq} we already
know that for $n\in [0.5,0.6)\supset L:=[0.5,0.547]$,   the limit
cycle  would be unique, hyperbolic and repeller. By the
Bendixson-Poincar\'{e} Theorem the above facts force the existence of
another limit cycle and so a contradiction. It is straightforward
that the existence of this positive invariant region is not
compatible with the existence of a polycycle connecting both saddle
points.

To construct $\mathcal R$ we consider a function
$V_2(x,y)=g_0(y)+g_1(y)x+g_2(y)x^2$, with $g_0$ and $g_1$ as
in~\eqref{g0g1} and $g_{2}$ an even polynomial function of degree 12
of the form
$$g_2(y)=1+\sum_{k=1}^6a_{2k}y^{2k},$$
to be determined. By statement (ii) of Lemma \ref{LeVM}, the
function $M_2$, given in \eqref{Mode3}, associated to this $V_2$ and
$k=2/3$ is of the form $M_2=\phi(y)x+\psi(y)$, where $\phi(y)$ and
$\psi(y)$ are polynomials in the variable $y$ whose coefficients
depend on the unknowns $a_{2k}$ with $k=1\ldots 6$.

We fix  $a_4$ and $a_6$ in such a way that $\phi(y)$ has a zero at
$y=0$ of multiplicity nine; we get the value of $a_8$ by imposing
that $V_2$ vanishes at the two saddle points; the values of $a_2$
and $a_{10}$ are chosen so that the curve $V_2=0$ is tangent to both
separatrices  at the saddle points of the system. Finally, after
experimenting with several values for $a_{12}$ and $m$, so that the
region with boundary $\{V_2=0\}$ is positively invariant, we fix
$a_{12}=-157(10m-3)(35m+3)/44550000$.

The region $\mathcal R$ will be the bounded connected component of
$\mathbb{R}^2\setminus\{V_2=0\}$ containing the origin, see Figure
\ref{funcionV} $(a)$.

\begin{figure}[h]
\begin{center}
\begin{tabular}{cc}
\epsfig{file=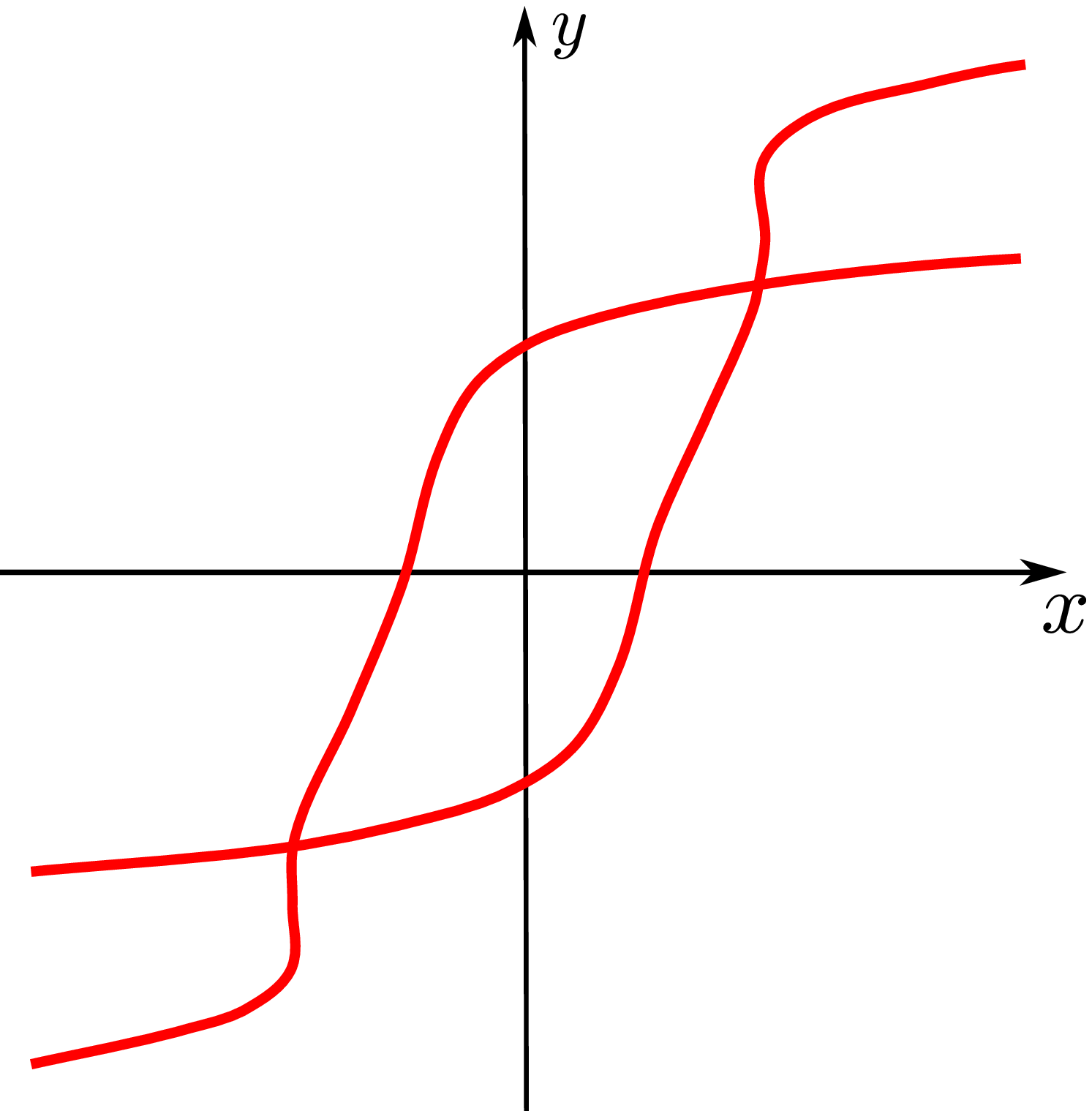,width=3.5cm}\qquad\qquad&\qquad\qquad
\epsfig{file=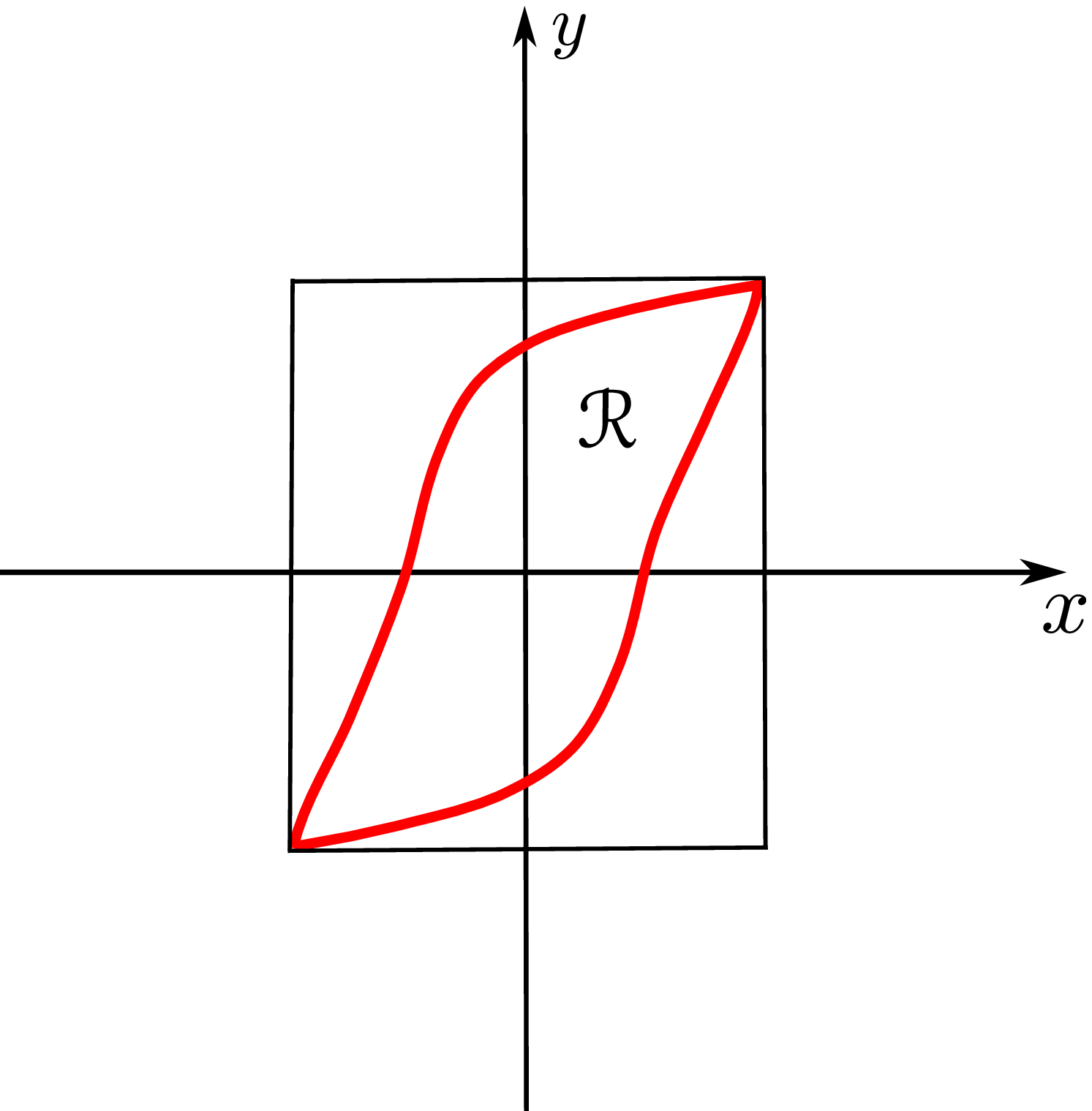,width=3.5cm}\\
(a)\qquad\qquad & \qquad\qquad (b)
\end{tabular}
\caption{Positively invariant region $\mathcal R$ with boundary
$\{V_2=0\}$.} \label{funcionV}
\end{center}
\end{figure}

We need to prove that the curve $\{V_2=0\}\cap\mathcal{S}$ (see
Figure~\ref{funcionV} (b)) is such that the vector field $X$ points
in on all its points. We introduce the new parameter $m=n^2$ and we
compute $\dot{V_2}=\langle\nabla V_2,X\rangle$ and
\begin{equation}\label{re}
\Res(V_2,\dot{V_2},x)=
\frac{y^8(ny^2-1)^4(P_{n,12}(y))^3P_{n,36}(y)}{n^{28}(120n^3+113n^2-3)^6},
\end{equation}
where $P_{n,12}(y)$ and $P_{n,36}(y)$ are polynomials of degree 12
and 36, respectively, and  whose coefficients are polynomial
functions in the variable $n$.

Notice that since $m\in[0.5,0.547]$ then $n\in T:=[0.707,0.7396]$.
Since the denominator of \eqref{re} is positive, we  only need to
study its numerator.

Using once more Lemma~\ref{ll} and the same tools that in the
previous sections we prove that $P_{n,12}(y)\cdot P_{n,36}(y)$ is
positive for all $y\in(-1/n,1/n)$ and $n\in T.$ We omit the details.

Hence, we have proved that the numerator of $\Res(V_2,\dot{V_2},x)$ is non-negative
and it only vanishes on $y=0$ and $y=\pm n^{-1/2}$. Therefore the sets $\{V_2=0\}$
and $\{\dot V_2=0\}$ only can intersect on $\{y=0\}.$ Indeed, the sets
$\{V_2=0\}\cap\mathcal{S}\cap\{y=0\}$ and $\{\dot
V_2=0\}\cap\mathcal{S}\cap\{y=0\}$ coincide and have two points $(\pm \widehat
x(n),0)$ for each $n\in T.$ Studying the local Taylor expansions of $V_2(x,y)$ and
$\dot V_2(x,y)$ at these points we get that the respective curves $V_2(x,y)=0$ and
$\dot V_2(x,y)=0$ have at them a fourth order contact point and,  as a consequence,
$\dot V_2$ does not change sign on $\{V_2=0\}\cap\mathcal{S}$, as we wanted to
prove. That, on $\{V_2=0\}$, the vector field $X$ points in,  is a simple
verification. Hence the proof follows.\end{proof}

\section{Existence of polycyles}\label{se:hetero}

This section is devoted to prove that the phase portrait  (b) in
Figure~\ref{Sphere} can only  appear for finitely many values of
$m$. Notice that this phase portrait is precisely the only one
presenting a polycycle. As we have already explained, the main
difficulty is that we are dealing with a  family that is not a
SCFRVF. To see that the control of the existence of polycycles for
general polynomial 1-parameter families can be a non easy task, we
present a simple
 family for which a polycycle appears at least for two values of the parameter.

\begin{example}\label{exam} For $m=0$ and $m=1$, the planar systems
\begin{equation}\label{exem}
\left\{\begin{array}{lll} \dot{x}= -2y+(3m-4)x+(4-2m)x^3+xy^2-x^5=P_m(x,y),
\\
\dot{y}= (4-m)x+xy^2-2mx^3-x^5=Q_m(x,y), \qquad m\in\mathbb{R}.
\end{array}\right.
\end{equation}
have a heteroclinic polycycle connecting the saddle points located at
$(\pm\sqrt{2-m},0)$.
\end{example}

\begin{proof} The above family has been cooked to have explicit algebraic
polycycles. Consider the family of algebraic curves $H_m(x,y)= y^2-(x^2+m-2)^2=0$
and compute
\[ W_m(x,y)=\langle\nabla H_m(x,y),(P_m(x,y),Q_m(x,y))\rangle.\] Doing the resultant  with respect to $x$  of $W_m$
and $H_m$ we obtain
\[
\Res(W_m(x,y),H_m(x,y),x)=m^4(1-m)^4y^4R(y,m),
\]
where $R$ is a polynomial of degree 4 in both variables, $m$ and $y$. This implies
that for $m=0$ and $m=1$ the algebraic curve $H_m(x,y)=0$ is invariant by the flow
of~\eqref{exem}. These sets coincide with the invariant manifolds of the saddle
points  $(\pm\sqrt{2-m},0)$ and contain the corresponding heteroclinic polycycles.
\end{proof}

We have simulated the phase portraits of~\eqref{exem} for several values of  $m$
and it seems that no polycycles appear for other values of $m$. In any case, the
example shows the differences between SCFRVF, for which as we have discussed in
Subsection~\ref{ss:nscf}, the polycycle usually appears for a single value of the
parameter, and families that are not SCFRVF.

Let us continue our study of system~\eqref{sism}. We  denote by
${\bf p}_m^{\pm}=(\pm m^{-1/4},\pm m^{-1/4})$ the two saddle points of
the system.

\begin{proposition}\label{ppp} Let $(0,y^s(m))$ be the first cut of the stable manifold of
${\bf p}_m^+$ with the $Oy^+$-axis. Similarly, let $(0,y^u(m))$ be
the first cut of the unstable manifold of ${\bf p}_m^-$ with the
same axis, see Figure~\ref{anal} (a). Then the function
$\delta(m):=y^s(m)-y^u(m)$ is an analytic function.
\end{proposition}

\begin{figure}[h]
\begin{center}
\begin{tabular}{cc}
\epsfig{file=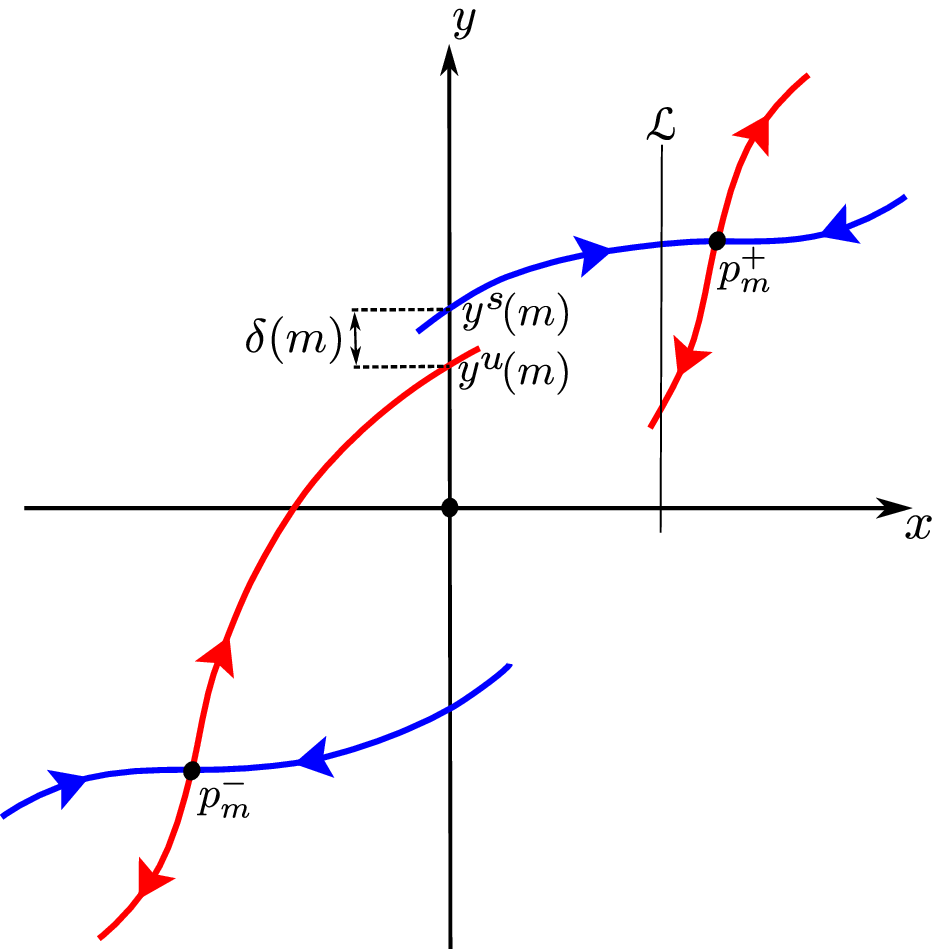,height=5.5cm}\qquad&\qquad
\epsfig{file=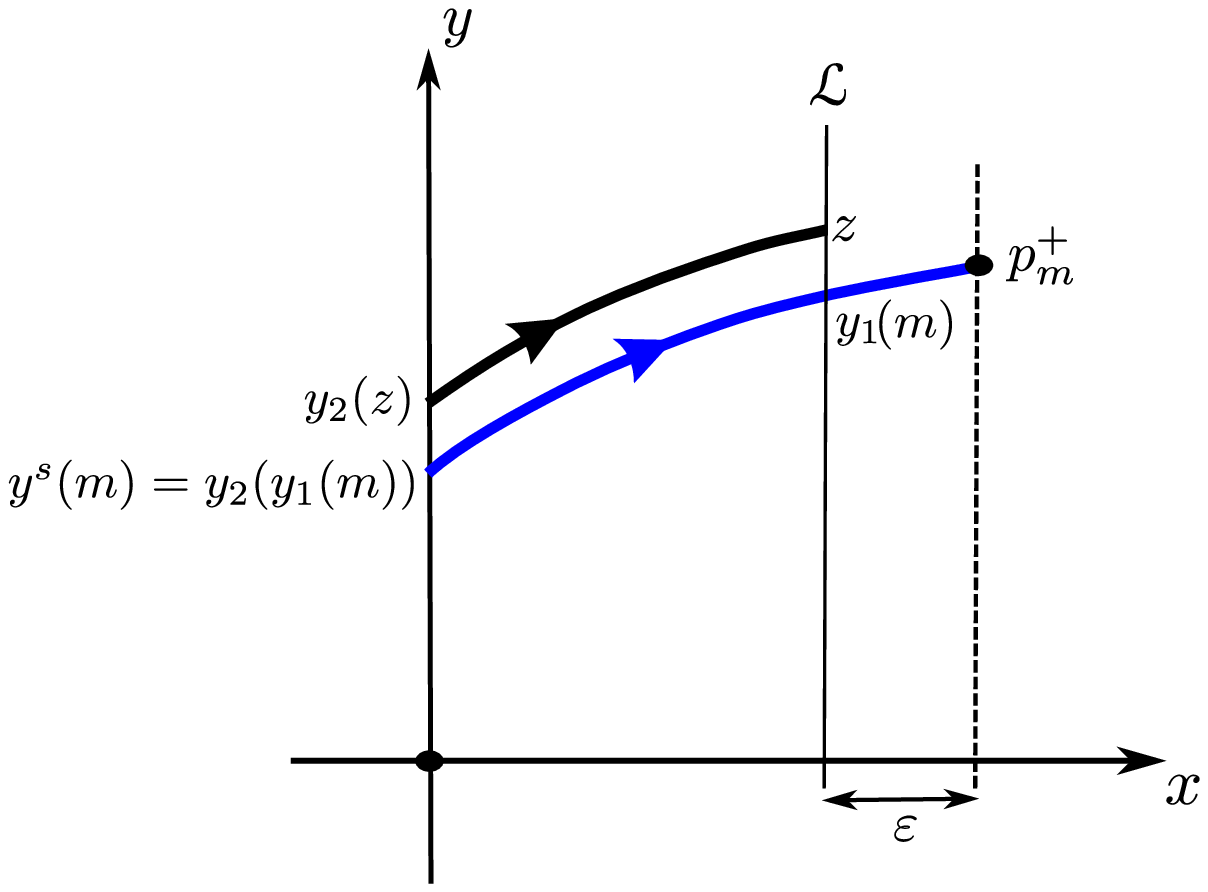,height=5.5cm}\\
(a) & \qquad (b)
\end{tabular}
\caption{Definition of the maps $\delta(m)$ and $y^s(m)$ in Proposition~\ref{ppp}.}
\label{anal}
\end{center}
\end{figure}

\begin{proof} This result is a consequence of the tools introduced in \cite{Pe}. We
only give the key points of that proof.

Fix a value $\widehat m$ for which $\delta(m)$ is defined. Simply
because the $Oy^+$ is transversal for the flow, the function
$\delta$ is well defined in a neighborhood of $\widehat m.$ It is
clear that it suffices to prove that $y^s(m)$ is analytic at
$m=\widehat m,$ because the $y^u(m)$ can be studied similarly. To
prove this fact we will write the map $y^s(m)$ as the composition of
two analytic maps.

Consider a vertical straight line
$\mathcal{L}:=\{(x,y)\,:\,x=\widehat m^{-1/4}-\varepsilon\}$, for
$\varepsilon>0$ small enough. Denote by $(
\widehat{m}^{-1/4}-\varepsilon ,y_1(m))$ the first cutting point of
the stable manifold of ${\bf p}_m^+$ with this line. Because
$\mathcal{L}$ is close enough to the saddle point it can be seen
that the local stable manifold cuts this line transversally.
Moreover, the tools given in \cite{Pe} prove that $y_1(m)$ is
analytic at $m=\widehat m,$ because of the hyperbolicity of the
saddle point. Next, consider the orbit starting on $\mathcal{L}$
with $y$-coordinate $y_1(\widehat m)$. In backward time, this orbit
cuts also transversally the $Oy^+$-axis at the point with
$y$-coordinate $y^s(\widehat m)$ and needs a finite time to arrive
to this point see Figure~\ref{anal} (b). Because of the
transversality to both lines, and the finiteness of the needed time
for going from one to the other, it is clear that the map $y_2(z)$
induced by the flow of the system between $\mathcal L$ and  the
$Oy^+$-axis is analytic  at $z=y_1(\widehat m)$. Since $y^s (m)=
y_2(y_1(m))$, the result follows.
\end{proof}

\begin{proof}[Proof of (iii) of Theorem~\ref{mteo}]
Notice that each value of $m$ that is  a zero of the map
$\delta(m),$ introduced in Proposition~\ref{ppp}, corresponds to a
system~\eqref{sism}  with a polycycle, i.e.
$\mathcal{M}=\{m\in(0.547,0.6)\,:\, \delta(m)=0\}.$ From
Proposition~\ref{547} we know that $\delta(0.547)>0$ and from
Proposition~\ref{35} that $\delta(0.6)<0.$ Hence the set
$\mathcal{M}$ is non-empty. Finally, because of the non-accumulation
property of the zeros of  analytic functions, the finiteness of
$\mathcal{M}$ follows.
\end{proof}

%
%

\section{Proof of Theorem \ref{mteo}}\label{se:teo}

The proof of Theorem~\ref{mteo} simply consists in gluing the corresponding results
proved along the paper. More concretely:
\begin{itemize}

\item The non existence of limit cycles and polycycles when
$m\in(-\infty,0.547]\cup[3/5,\infty)$ is given in the following results:
\begin{itemize}
\item For $m\in(-\infty,0]$, trivially in the introduction.
\item For $m\in(0,9/25]$ in Proposition~\ref{nc},
\item For $m\in(9/25,1/2)$ in Proposition~\ref{925},
\item For $m\in[1/2,0.547]$ in Proposition~\ref{547},
\item For  $m\in[3/5,\infty)$ in Proposition~\ref{35}.
\end{itemize}

\item The existence of  at most
 one limit cycle and one polycycle  when $m\in[1/2,3/5)$, the fact that
 they never  coexist, and the  hyperbolicity and  instability of the limit cycle,
 in Proposition~\ref{uniq}.

\item The phase portraits of the system in the Poincar\'{e} disc and the study of the
origin, in Subsection~\ref{ss:retrato} and Section~\ref{ss:monod}, respectively.

\item The proof of the existence of the phase portrait (b) in Figure~\ref{Sphere},
 only for finitely many values of $m$, in Section~\ref{se:hetero}.
\end{itemize}

\subsection*{Acknowledgements}
The first two authors are  supported by the MICIIN/FEDER  grant
number MTM2008-03437 and the Generalitat de Catalunya grant number
2009-SGR 410. The first author is also supported by the grant
AP2009-1189.

\end{document}